\documentclass[preprint,12pt]{elsarticle}

\usepackage{amsmath,amsthm,mathtools}
\usepackage{fullpage}
\usepackage{amssymb}
\usepackage{comment}
\usepackage{tikz}
\usepackage{float}
\usepackage{graphicx}
\usepackage{caption}
\usepackage{subcaption}

\newcommand{\wVdual}{\omega_{\mathcal{V}^*}}
\newcommand{\wFdual}{\omega_{\mathcal{F}^*}}

\newcommand{\bG}{\mathbb{G}}
\newcommand{\bH}{\mathbb{H}}
\newcommand{\ba}{\backslash}
\newcommand{\Tt}{\mathcal{T}}
\newcommand{\T}{\boldsymbol{T}}
\newcommand{\V}{\mathcal{V}}
\newcommand{\F}{\mathcal{F}}
\newcommand{\wV}{\omega_{\mathcal{V}}}
\newcommand{\wF}{\omega_{\mathcal{F}}}
\newcommand{\pG}{\underline{\mathbb{G}}}
\newcommand{\n}{\eta(e)}
\newcommand{\m}{\mu(e)}

\newcommand{\mgd}{\mu_{\pG^*}(e)}
\newcommand{\np}{\eta_{\pG}(e)}

\newtheorem{theorem}{Theorem}[section]
\newtheorem{corollary}[theorem]{Corollary}
\newtheorem{lemma}[theorem]{Lemma}
\newtheorem{proposition}[theorem]{Proposition}
\theoremstyle{definition}
\newtheorem{definition}[theorem]{Definition}
\newtheorem{example}[theorem]{Example}

\begin{document}

\begin{frontmatter}
\title{Deletion-Contraction and the Surface Tutte Polynomial}

\author[rh]{Iain Moffatt}
\ead{iain.moffatt@rhul.ac.uk}
\author[rh]{Maya Thompson}
\ead{maya.thompson.2019@live.rhul.ac.uk}

\affiliation[rh]{organization={Department of Mathematics},
           addressline={Royal Holloway}, 
             addressline={University of London}, 
            city={Egham}, 
            postcode={{TW20~0EX}},
           country={United Kingdom}}

\begin{abstract}
In this paper we unify two families of topological Tutte polynomials.
The first family is that coming from the surface Tutte polynomial, a polynomial that arises in the theory of local flows and tensions. 
The second family arises from the canonical Tutte polynomials of Hopf algebras. 
Each family includes the Las~Vergnas, Bollob\'as--Riordan, and Krushkal polynomials. 
As a consequence we determine a deletion--contraction definition of the surface Tutte polynomial and recursion relations for the number of local flows and tensions in an embedded graph.

\end{abstract}

\end{frontmatter}


\section{Introduction}\label{sint}
By way of motivation,  we begin with a brief discussion of the classical Tutte polynomial
\begin{equation}\label{tp}
T(G;x,y):=\sum_{A\subseteq E} (x-1)^{r(E)-r(A)}(y-1)^{|A|-r(A)} 
\end{equation}
  of a graph $G=(V,E)$,  where  $r(A)$ is the rank of the spanning subgraph on $A$. (See, for example,~\cite{zbMATH01179517, zbMATH00067324, zbMATH05828840, zbMATH07553843,zbMATH00437298} for  background on the Tutte polynomial.) 
Tutte introduced his polynomial, initially named the dichromate, as a bivariate generalisation of both the flow and chromatic polynomials~\cite{zbMATH03087501,zbMATH02075368}. 
Indeed the number of nowhere-zero $\mathbb{Z}_k$-flows  and the number of proper $k$-colourings of $G$  can be recovered from the Tutte polynomial through the evaluations $T(G;0,1-k)$ and $T(G;1-k,0)$, respectively. 
(See, for example, \cite[Section~3.4]{zbMATH07553843} for definitions and details.)
This can be seen to be a consequence of the \emph{Universality Property} of the Tutte polynomial, which gives that any graph parameter $U$ that satisfies the deletion-contraction relations
\begin{equation}\label{tp2}
U(G) = \begin{cases} 
x\,U(G\backslash e) & \text{if $e$ is a bridge,} \\
y\,U(G/ e) & \text{if $e$ is a loop,} \\
a\,U(G\backslash e) + b\,U(G\backslash e) & \text{if $e$ is an ordinary edge,} \\ \alpha^{|V|} &\text{if $G$ is edgeless};
\end{cases}
\end{equation}
 for some parameters $\alpha,a,b,x,y$, is necessarily an evaluation of the Tutte polynomial (see, for example,~\cite{zbMATH01179517,zbMATH07553843} for full details). Taking $\alpha=a=b=1$ and $x$ and $y$ to be variables in~\eqref{tp2} gives the \emph{deletion-contraction relations} for the Tutte polynomial, and $T(G;x,y)$ is uniquely defined through these relations.    The importance of the Universality Property is that it shows the Tutte polynomial encapsulates all graph deletion-contraction invariants, and goes towards explaining the importance of the polynomial in mathematics, why it appears in areas such as knot theory and statistical physics,  and why it captures such a wide and diverse range of graph parameters. Of particular relevance here is that deletion-contraction, flows and colourings are inextricably linked. 

Notice that Equation~\eqref{tp2} provides a recursive definition for the graph polynomial where the initial conditions are given by its values on edgeless graphs. In general, following common usage in the area, we say that a (topological) graph polynomial has a \emph{full} deletion-contraction definition (or \emph{full} recursive definition) if it is uniquely defined by a set of deletion-contraction relations together with its values on edgeless graphs. Thus edgeless graphs provide the initial conditions of the deletion-contraction recurrence relations for the polynomial. 

 \medskip

In this paper we are interested in analogues of the Tutte polynomial for graphs embedded in \emph{orientable} surfaces. (In subsequent sections we shall realise embedded graphs as ribbon graphs.) Considerable attention has been paid to such \emph{topological Tutte polynomials} over recent years. Much of this recent interest grew from Bollob\'as and Riordan's papers~\cite{bollobasriordanpoly,zbMATH01801590}, although the Las~Vergnas polynomial predates it~\cite{zbMATH03722664}. 
A challenge when studying topological Tutte polynomials is that, unlike the classical case for graphs, there are many different analogues for the Tutte polynomial~\cite{bollobasriordanpoly,zbMATH01801590,maps1, maps2, HM,KMT,krushkalpoly,zbMATH03722664,zbMATH06951578}. 
Here we tie together two strands in the development of topological Tutte polynomials, and as a consequence find a graph  polynomial that has a full deletion-contraction definition and that subsumes all these topological Tutte polynomials.

Initially, topological Tutte polynomials were defined through  state-sum formulas analogous to~\eqref{tp}. This is the case for the various polynomials introduced in~\cite{bollobasriordanpoly,zbMATH01801590,maps1, maps2, krushkalpoly,zbMATH03722664}. A difficulty with these polynomials is that, unlike  the classical Tutte polynomial, their deletion-contraction relations do not apply to all edges of graphs, and therefore do not provide a full recursive definition of the polynomial.

A systematic approach to constructing graph polynomials which have full recursive definitions  was taken in~\cite{HM,KMT}. There the idea was to start with some concept of deletion and contraction and, via Hopf algebras, use them to  generate a graph polynomial. 
 The resulting polynomials, called \emph{canonical Tutte polynomials},  have a Universality Property akin to~\eqref{tp2} and deletion-contraction recurrence relations with edgeless graphs providing the initial conditions.
 It was shown in~\cite{HM,KMT} that there is a \emph{canonical Tutte polynomial of graphs in pseudo-surfaces} that has a full recursive deletion-contraction definition, has a Universality Property, and contains all of Las~Vergnas', Bolloba\'s and Riordan's and Krushkal's topological Tutte polynomials as specialisations. (See~\cite{zbMATH05937085} for a non-full ``Recipe Theorem'' for the Bolloba\'s--Riordan polynomial.)

After that work, a different approach to constructing topological Tutte polynomials was taken by Goodall et al. in~\cite{maps1, maps2}. Their approach to topological Tutte polynomials was motivated by Tutte's original construction of $T(G;x,y)$ as a bivariate generalisation of both the flow and chromatic polynomials. 
  Flows and tensions of graphs (the latter corresponding to vertex colourings)
can be extended to embedded graphs in the form of \emph{local flows} and \emph{local
tensions} taking values in a not necessarily Abelian group (the cyclic order of
edges around a vertex or face given by the embedding ensuring they are well
defined). Enumerating them gives rise to analogues for embedded graphs of
the chromatic and flow polynomial for graphs. The authors of~\cite{maps1} were able to combine these two polynomials to construct a multivariate polynomial of graphs in surfaces, called the \emph{surface Tutte polynomial}. Just as the Tutte polynomial of a graph yields by evaluation the number of nowhere-zero flows and tensions, so the surface Tutte polynomial of an embedded graph yields by evaluation the number of nowhere-identity local flows and tensions.
 However, the surface Tutte polynomial was defined through a state-sum analogous to~\eqref{tp} and it does not have a full deletion-contraction definition or Universality Property akin to Equation~\eqref{tp2}.

In this paper we unify the approaches of Krajewski et al.~\cite{HM,KMT} and Goodall et al.~\cite{maps1}. We show that \emph{all} of the above polynomials belong to a single family and can be recovered from a single graph polynomial. 
We do this by extending the domain of the surface Tutte polynomial to objects we call ``packaged ribbon graphs'' (introduced in Section~\ref{anj}). 
Then in Section~\ref{shjk} we introduce a Tutte polynomial for packaged ribbon graphs (Definition~\ref{fhasj}) and show that it has a full recursive deletion-contraction definition with initial values given by edgeless ribbon graphs (Theorem~\ref{dcthm}), and that it has a Universality Property analogous to that for the classical Tutte polynomial (Theorem~\ref{thun}).  
 We show that our graph polynomial specialises to both the surface Tutte polynomial of~\cite{maps1} \emph{and} the canonical Tutte polynomials~\cite{HM,KMT} of graphs in surfaces (Theorem~\ref{asg}). It follows that our polynomial includes \emph{all} of the topological Tutte polynomials discussed above (see Figure~\ref{rhj}). Finally, in Section~\ref{daha} we apply our results to give a recursive way to count the numbers of local flows and local tensions in a graph in a surface.

\section{Packaged ribbon graphs}\label{anj}
As noted above, our approach is to unify the two families of topological Tutte polynomials by extending their domain to a class of objects we call  \emph{packaged ribbon graphs}. In this section we introduce packaged ribbon graphs and operations of deletion, contraction and duality on them. 
Packaged ribbon graphs consist of ribbon graphs equipped with a partition on their sets of vertices and sets of boundary components. Furthermore, each block of a partition has a weight. It is well-known that ribbon graphs describe graphs cellularly embedded in surfaces (see e.g., \cite{graphsonsurfaces,zbMATH01624430}). Similarly, packaged ribbon graphs describe graphs (not necessarily cellularly) embedded in pseudo-surfaces.

\subsection{A review of ribbon graphs}
We begin with a brief overview of standard ribbon graph terminology and notation.  Further details on ribbon graphs can be found in \cite{graphsonsurfaces}, and a reader familiar with them may safely skip this subsection. Note that in this paper we assume all our ribbon graphs are orientable (in general, they can be non-orientable).

\begin{definition}
    A \emph{ribbon graph} $\bG=(V,E)$ is an orientable surface with boundary represented by the union of a set of discs $V$, called the vertices, and another set of discs $E$, called the edges, satisfying the following constraints:
   \begin{enumerate}
        \item A vertex and edge are either disjoint or meet in an arc.
        \item Each such arc lies on the boundary of just one vertex and one edge.
        \item Every edge contains two such arcs.
    \end{enumerate}
\end{definition}
In the definition of a ribbon graph, note that the orientable surface with boundary need not be connected.

Given a ribbon graph $\bG$, we use $v(\bG)$, $e(\bG)$, $f(\bG)$ and $k(\bG)$ to denote the number of vertices, edges, boundary components and connected components respectively. 
A ribbon graph $\bG$ has  \emph{genus} $g(\bG):=\frac{1}{2}(2k(\bG)-f(\bG)+e(\bG)-v(\bG))$; \emph{rank} $r(\bG):=v(\bG)-k(\bG)$; and \emph{nullity} $n(\bG):=e(\bG)-r(\bG)$. 
Where it makes sense, we  use this terminology and notation for graphs too.

Let $\bG=(V,E)$ be a ribbon graph and $e\in E$ be an edge. The ribbon graph obtained by \emph{deleting} $e$ is denoted by $\bG\ba e:=(V,E\setminus e)$ (we use the convention of denoting singleton sets by the element they contain). 
Similarly,  the ribbon graph obtained by deleting an isolated vertex $v$ (i.e. a vertex with no incident edges) is denoted by $\bG\ba v:=(V\setminus v,E)$. When deleting edges one after another the order in which we do so does not matter, and likewise when deleting vertices. This allows us to extend the operation of deletion to subsets of edges and subsets of vertices.
We say that a ribbon graph $\bH$ is a \emph{ribbon subgraph} of a ribbon graph $\bG$ if it can be obtained from $\bG$ by deleting a subset of edges (possibly empty) and then deleting a subset of isolated vertices (possibly empty).
Let $\bH=(V',A)$ be a ribbon subgraph of $\bG=(V,E)$. We say $\bH$ is a \emph{spanning ribbon subgraph} of $\bG$ if $V'=V$ and in this case we denote $\bH$ by $\bG|A$. Note that $\bG|A=\bG\backslash A^c$ where $A^c=E\setminus A$.

\begin{definition} 
 Let $\bG=(V,E)$ be a ribbon graph. Let $e\in E$ and  $u$ and $v$ be its incident vertices, which are not necessarily distinct. Then  $\bG/e$ denotes the ribbon graph obtained as follows. Consider the boundary component(s) of $e\cup u \cup v$ as curves on $\bG$. For each resulting curve, attach a disc, which will form a vertex of $\bG/e$, by identifying its boundary component with the curve. Delete $e$, $u$ and $v$ from the resulting complex.
We say that $\bG/e$ is obtained from $\bG$ by \emph{contracting} $e$.
\end{definition}

For a depiction of ribbon graph contraction see Table~\ref{Con} and disregard the colouring and weighting, and similarly for ribbon graph deletion and Table~\ref{Del}.

The \emph{(geometric) dual} $\bG^*$ of a ribbon graph $\bG=(V,E)$ can be constructed in the following way. 
Recall that, topologically, a ribbon graph is a surface with boundary; we cap off the holes using a set of discs, denoted by $V(\bG^*)$, to obtain a surface without boundary. The geometric dual of $\bG$ is the ribbon graph $\bG^* = (V(\bG^*),E)$. Observe that the edges of $\bG$ and $\bG^*$ are identical. The only change is which arcs on their boundaries meet vertices.

\subsection{Packaged ribbon graphs and operations on them}

We shall now introduce packaged ribbon graphs, which are our main object of study. Recall that, we assume every ribbon graph is orientable including those within packaged ribbon graphs.

\begin{definition}
   A  \emph{packaged ribbon graph} is a tuple $\pG=(\bG,\V,\wV,\F,\wF)$ where: 
    \begin{enumerate}
        \item $\bG=(V,E)$ is a ribbon graph with set of boundary components $F$;
        \item $\V$ is a partition of $V$, and $\F$ is a partition of $F$;
        \item $\wV:\V\to \mathbb{N}_0$ is a \emph{weighting} on the blocks of $\V$, and $\wF:\F\to\mathbb{N}_0$ is a \emph{weighting} on the blocks of $\F$.
    \end{enumerate}
\end{definition}

For each vertex $v\in V$, we denote its block in the partition $\V$ by $[v]$. Similarly, for each boundary component $f\in F$, we denote its block in the partition $\F$ by $[f]$. (Note that in this paper the term ``block" will always refer to the block of a partition and never a block of a graph.)

A packaged ribbon graph consists of a ``coloured ribbon graph'' as defined by Huggett and Moffatt~\cite[Definition~11]{HM} in which the vertex colours and boundary component colours are weighted. In practice, it can be easier to visualise the partitions $\V$ and $\F$ as colourings on the vertices and boundary components respectively. In which case, the block weightings become colour weightings.
For this reason, the figures in this paper will use ribbon graphs with weighted vertex and boundary component colourings to depict packaged ribbon graphs. The weightings $\wV$ and $\wF$ will be represented by a coloured tuple where each integer corresponds to the weight function on that colour. An example can be found in Figure~\ref{Packaging}.

\medskip
Deletion and contraction are defined for packaged ribbon graphs following their definitions for coloured ribbon graphs from~\cite[Definition~13]{HM} and specifying how the weights change, as below.
 The reader may find it helpful to consult Tables~\ref{Del} and~\ref{Con} while reading the definitions.

\begin{definition}
    Let $\pG=(\bG,\V,\wV,\F,\wF)$ be a packaged ribbon graph, $F$ be its set of boundary components, and $e$ be an edge in $\bG$. We use $\pG\ba e$ to denote the packaged ribbon graph $(\bG\ba e,\V',\wV',\F',\wF')$ obtained by \emph{deleting} $e$, and define it as follows. Recalling $\bG$ and $\bG\ba e$ have the same vertices,  we set $\V':=\V$ and $\wV':=\wV$. For the boundary components, $\F'$ and $\wF'$ are defined as follows.
    \begin{enumerate}
        \item\label{dcase1} If, in $\bG$, the edge $e$ intersects one boundary component $f\in F$ twice, then delete $e$ from $\bG$ and let $f'$, $g'$ denote the two boundary components formed by the deletion. Obtain the partition $\F'$ from $\F$ by replacing $[f]$ with $[f']:=[f]\cup\{f',g'\}\setminus \{f\}$. The other blocks are unchanged. Define 
        \[\wF'([x]):=
        \begin{cases}
        \wF([x])+1 & \text{ if } [x]=[f'],\\
        \wF([x]) & \text{ if } [x]\neq [f'].
        \end{cases}\]
        
        \item\label{dcase2} If, in $\bG$,  the edge $e$ intersects two boundary components $f,g\in F$ and $[f]=[g]$, then delete $e$ from $\bG$ and let $f'$ denote the boundary component formed by the deletion. Obtain the partition $\F'$ from $\F$ by replacing $[f]$ with $[f']:=[f]\cup\{f'\}\setminus \{f, g\}$. The other blocks are unchanged. Define
        \[\wF'([x]):=
        \begin{cases}
        \wF([x])+1 & \text{ if } [x]=[f'],\\
        \wF([x]) & \text{ if } [x]\neq [f'].
        \end{cases}\]
        
        \item\label{dcase3} If, in $\bG$,  the edge $e$ intersects two boundary components $f,g\in F$ and $[f]\neq [g]$, then delete $e$ from $\bG$ and let $f'$ denote the boundary component formed by the deletion. Obtain the partition $\F'$ from $\F$ by removing both $[f]$ and $[g]$, and inserting $[f']:=[f]\cup[g]\cup \{f'\}\setminus\{f,g\}$. The other blocks are unchanged. Define 
        \[\wF'([x]):=
        \begin{cases}
        \wF([f])+\wF([g]) & \text{ if } [x]=[f'],\\
        \wF([x]) & \text{ if } [x]\neq [f'].
        \end{cases}\]
    \end{enumerate}
    
\end{definition}

\begin{table}
    \centering
    \begin{tabular}{| c | c |}
        \hline
        $\pG$ & $\pG\ba e$ \\
        \hline
         \tikzset{every picture/.style={line width=0.5pt}} 

\begin{tikzpicture}[x=0.55pt,y=0.55pt,yscale=-1,xscale=1]

\draw  [fill={rgb, 255:red, 0; green, 0; blue, 0 }  ,fill opacity=1 ] (59,55) .. controls (59,41.19) and (70.19,30) .. (84,30) .. controls (97.81,30) and (109,41.19) .. (109,55) .. controls (109,68.81) and (97.81,80) .. (84,80) .. controls (70.19,80) and (59,68.81) .. (59,55) -- cycle ;
\draw  [fill={rgb, 255:red, 0; green, 0; blue, 0 }  ,fill opacity=1 ] (199,54) .. controls (199,40.19) and (210.19,29) .. (224,29) .. controls (237.81,29) and (249,40.19) .. (249,54) .. controls (249,67.81) and (237.81,79) .. (224,79) .. controls (210.19,79) and (199,67.81) .. (199,54) -- cycle ;
\draw [color={rgb, 255:red, 223; green, 83; blue, 107 }  ,draw opacity=1 ][line width=2.25]    (105.9,43.91) -- (200.9,43.91) ;
\draw [color={rgb, 255:red, 223; green, 83; blue, 107 }  ,draw opacity=1 ][line width=2.25]    (106.7,64.71) -- (201.7,64.71) ;
\draw [color={rgb, 255:red, 223; green, 83; blue, 107 }  ,draw opacity=1 ][line width=2.25]    (32.1,18.91) -- (64.5,39.31) ;
\draw [color={rgb, 255:red, 223; green, 83; blue, 107 }  ,draw opacity=1 ][line width=2.25]    (244.1,69.31) -- (276.5,89.71) ;
\draw [color={rgb, 255:red, 223; green, 83; blue, 107 }  ,draw opacity=1 ][line width=2.25]    (63.7,69.4) -- (31.3,89.8) ;
\draw    (59.82,60.43) -- (27.42,80.83) ;
\draw    (281.07,28.6) -- (248.67,49) ;
\draw    (248.1,60.01) -- (280.5,80.41) ;
\draw    (27.1,27.93) -- (59.5,48.33) ;
\draw  [draw opacity=0][line width=2.25]  (107.57,63.35) .. controls (104.13,73.05) and (94.88,80) .. (84,80) .. controls (75.08,80) and (67.25,75.33) .. (62.83,68.3) -- (84,55) -- cycle ; \draw  [color={rgb, 255:red, 223; green, 83; blue, 107 }  ,draw opacity=1 ][line width=2.25]  (107.57,63.35) .. controls (104.13,73.05) and (94.88,80) .. (84,80) .. controls (75.08,80) and (67.25,75.33) .. (62.83,68.3) ;  
\draw  [draw opacity=0][line width=2.25]  (63.54,40.62) .. controls (68.07,34.2) and (75.54,30) .. (84,30) .. controls (94.27,30) and (103.09,36.19) .. (106.94,45.05) -- (84,55) -- cycle ; \draw  [color={rgb, 255:red, 223; green, 83; blue, 107 }  ,draw opacity=1 ][line width=2.25]  (63.54,40.62) .. controls (68.07,34.2) and (75.54,30) .. (84,30) .. controls (94.27,30) and (103.09,36.19) .. (106.94,45.05) ;  
\draw  [draw opacity=0][line width=2.25]  (200.48,45.5) .. controls (203.96,35.88) and (213.18,29) .. (224,29) .. controls (232.72,29) and (240.4,33.47) .. (244.88,40.24) -- (224,54) -- cycle ; \draw  [color={rgb, 255:red, 223; green, 83; blue, 107 }  ,draw opacity=1 ][line width=2.25]  (200.48,45.5) .. controls (203.96,35.88) and (213.18,29) .. (224,29) .. controls (232.72,29) and (240.4,33.47) .. (244.88,40.24) ;  
\draw [color={rgb, 255:red, 223; green, 83; blue, 107 }  ,draw opacity=1 ][line width=2.25]  [dash pattern={on 6.75pt off 4.5pt}]  (32.1,18.91) .. controls (-7.9,29.36) and (-7.38,77.54) .. (32.28,89.41) ;
\draw  [draw opacity=0][line width=2.25]  (244.8,67.87) .. controls (240.32,74.58) and (232.68,79) .. (224,79) .. controls (213.28,79) and (204.14,72.26) .. (200.59,62.79) -- (224,54) -- cycle ; \draw  [color={rgb, 255:red, 223; green, 83; blue, 107 }  ,draw opacity=1 ][line width=2.25]  (244.8,67.87) .. controls (240.32,74.58) and (232.68,79) .. (224,79) .. controls (213.28,79) and (204.14,72.26) .. (200.59,62.79) ;  
\draw [color={rgb, 255:red, 223; green, 83; blue, 107 }  ,draw opacity=1 ][line width=2.25]  [dash pattern={on 6.75pt off 4.5pt}]  (276.1,19) .. controls (316.1,29.44) and (315.58,77.62) .. (275.92,89.5) ;
\draw [color={rgb, 255:red, 223; green, 83; blue, 107 }  ,draw opacity=1 ][line width=2.25]    (276.1,19) -- (244,39.16) ;
\draw [color={rgb, 255:red, 223; green, 83; blue, 107 }  ,draw opacity=1 ][line width=2.25]  [dash pattern={on 6.75pt off 4.5pt}]  (276.1,19) .. controls (316.1,29.44) and (315.58,77.62) .. (275.92,89.5) ;

\draw (148,16.9) node [anchor=north west][inner sep=0.75pt]  [font=\large]  {$\textcolor[rgb]{0.87,0.33,0.42}{n}$};

\end{tikzpicture} & \tikzset{every picture/.style={line width=0.5pt}} 

\begin{tikzpicture}[x=0.55pt,y=0.55pt,yscale=-1,xscale=1]

\draw  [fill={rgb, 255:red, 0; green, 0; blue, 0 }  ,fill opacity=1 ] (59,56) .. controls (59,42.19) and (70.19,31) .. (84,31) .. controls (97.81,31) and (109,42.19) .. (109,56) .. controls (109,69.81) and (97.81,81) .. (84,81) .. controls (70.19,81) and (59,69.81) .. (59,56) -- cycle ;
\draw  [fill={rgb, 255:red, 0; green, 0; blue, 0 }  ,fill opacity=1 ] (199,55) .. controls (199,41.19) and (210.19,30) .. (224,30) .. controls (237.81,30) and (249,41.19) .. (249,55) .. controls (249,68.81) and (237.81,80) .. (224,80) .. controls (210.19,80) and (199,68.81) .. (199,55) -- cycle ;
\draw [color={rgb, 255:red, 223; green, 83; blue, 107 }  ,draw opacity=1 ][line width=2.25]    (32.1,19.91) -- (64.5,40.31) ;
\draw [color={rgb, 255:red, 223; green, 83; blue, 107 }  ,draw opacity=1 ][line width=2.25]    (244.1,70.31) -- (276.5,90.71) ;
\draw [color={rgb, 255:red, 223; green, 83; blue, 107 }  ,draw opacity=1 ][line width=2.25]    (62.83,69.3) -- (30.43,89.7) ;
\draw    (59.82,61.43) -- (27.42,81.83) ;
\draw    (281.07,29.6) -- (248.67,50) ;
\draw    (248.1,61.01) -- (280.5,81.41) ;
\draw    (27.1,28.93) -- (59.5,49.33) ;
\draw  [draw opacity=0][line width=2.25]  (63.72,41.37) .. controls (68.26,35.09) and (75.65,31) .. (84,31) .. controls (97.81,31) and (109,42.19) .. (109,56) .. controls (109,69.81) and (97.81,81) .. (84,81) .. controls (75.08,81) and (67.25,76.33) .. (62.83,69.3) -- (84,56) -- cycle ; \draw  [color={rgb, 255:red, 223; green, 83; blue, 107 }  ,draw opacity=1 ][line width=2.25]  (63.72,41.37) .. controls (68.26,35.09) and (75.65,31) .. (84,31) .. controls (97.81,31) and (109,42.19) .. (109,56) .. controls (109,69.81) and (97.81,81) .. (84,81) .. controls (75.08,81) and (67.25,76.33) .. (62.83,69.3) ;  
\draw  [draw opacity=0][line width=2.25]  (244.64,69.11) .. controls (240.14,75.68) and (232.57,80) .. (224,80) .. controls (210.19,80) and (199,68.81) .. (199,55) .. controls (199,41.19) and (210.19,30) .. (224,30) .. controls (232.72,30) and (240.4,34.47) .. (244.88,41.24) -- (224,55) -- cycle ; \draw  [color={rgb, 255:red, 223; green, 83; blue, 107 }  ,draw opacity=1 ][line width=2.25]  (244.64,69.11) .. controls (240.14,75.68) and (232.57,80) .. (224,80) .. controls (210.19,80) and (199,68.81) .. (199,55) .. controls (199,41.19) and (210.19,30) .. (224,30) .. controls (232.72,30) and (240.4,34.47) .. (244.88,41.24) ;  
\draw [color={rgb, 255:red, 223; green, 83; blue, 107 }  ,draw opacity=1 ][line width=2.25]  [dash pattern={on 6.75pt off 4.5pt}]  (32.1,19.91) .. controls (-7.9,30.36) and (-7.38,78.54) .. (32.28,90.41) ;
\draw [color={rgb, 255:red, 223; green, 83; blue, 107 }  ,draw opacity=1 ][line width=2.25]  [dash pattern={on 6.75pt off 4.5pt}]  (276.1,20) .. controls (316.1,30.44) and (315.58,78.62) .. (275.92,90.5) ;
\draw [color={rgb, 255:red, 223; green, 83; blue, 107 }  ,draw opacity=1 ][line width=2.25]    (276.1,20) -- (244,40.16) ;
\draw [color={rgb, 255:red, 223; green, 83; blue, 107 }  ,draw opacity=1 ][line width=2.25]  [dash pattern={on 6.75pt off 4.5pt}]  (276.1,20) .. controls (316.1,30.44) and (315.58,78.62) .. (275.92,90.5) ;

\draw (134.3,19.53) node [anchor=north west][inner sep=0.75pt]  [font=\large,color={rgb, 255:red, 223; green, 83; blue, 107 }  ,opacity=1 ]  {$\textcolor[rgb]{0.87,0.33,0.42}{n+1}$};

\end{tikzpicture} \\
        \hline
        \tikzset{every picture/.style={line width=0.5pt}} 

\begin{tikzpicture}[x=0.55pt,y=0.55pt,yscale=-1,xscale=1]

\draw  [fill={rgb, 255:red, 0; green, 0; blue, 0 }  ,fill opacity=1 ] (61,56) .. controls (61,42.19) and (72.19,31) .. (86,31) .. controls (99.81,31) and (111,42.19) .. (111,56) .. controls (111,69.81) and (99.81,81) .. (86,81) .. controls (72.19,81) and (61,69.81) .. (61,56) -- cycle ;
\draw  [fill={rgb, 255:red, 0; green, 0; blue, 0 }  ,fill opacity=1 ] (201,55) .. controls (201,41.19) and (212.19,30) .. (226,30) .. controls (239.81,30) and (251,41.19) .. (251,55) .. controls (251,68.81) and (239.81,80) .. (226,80) .. controls (212.19,80) and (201,68.81) .. (201,55) -- cycle ;
\draw [color={rgb, 255:red, 223; green, 83; blue, 107 }  ,draw opacity=1 ][line width=2.25]    (107.9,44.91) -- (202.9,44.91) ;
\draw [color={rgb, 255:red, 223; green, 83; blue, 107 }  ,draw opacity=1 ][line width=2.25]    (108.7,65.71) -- (203.7,65.71) ;
\draw [color={rgb, 255:red, 223; green, 83; blue, 107 }  ,draw opacity=1 ][line width=2.25]    (34.1,19.91) -- (66.5,40.31) ;
\draw [color={rgb, 255:red, 223; green, 83; blue, 107 }  ,draw opacity=1 ][line width=2.25]    (246.1,70.31) -- (278.5,90.71) ;
\draw [color={rgb, 255:red, 223; green, 83; blue, 107 }  ,draw opacity=1 ][line width=2.25]    (65.7,70.4) -- (33.3,90.8) ;
\draw    (61.82,61.43) -- (29.42,81.83) ;
\draw    (283.07,29.6) -- (250.67,50) ;
\draw    (250.1,61.01) -- (282.5,81.41) ;
\draw    (29.1,28.93) -- (61.5,49.33) ;
\draw  [draw opacity=0][line width=2.25]  (109.57,64.35) .. controls (106.13,74.05) and (96.88,81) .. (86,81) .. controls (77.08,81) and (69.25,76.33) .. (64.83,69.3) -- (86,56) -- cycle ; \draw  [color={rgb, 255:red, 223; green, 83; blue, 107 }  ,draw opacity=1 ][line width=2.25]  (109.57,64.35) .. controls (106.13,74.05) and (96.88,81) .. (86,81) .. controls (77.08,81) and (69.25,76.33) .. (64.83,69.3) ;  
\draw  [draw opacity=0][line width=2.25]  (65.54,41.62) .. controls (70.07,35.2) and (77.54,31) .. (86,31) .. controls (96.27,31) and (105.09,37.19) .. (108.94,46.05) -- (86,56) -- cycle ; \draw  [color={rgb, 255:red, 223; green, 83; blue, 107 }  ,draw opacity=1 ][line width=2.25]  (65.54,41.62) .. controls (70.07,35.2) and (77.54,31) .. (86,31) .. controls (96.27,31) and (105.09,37.19) .. (108.94,46.05) ;  
\draw  [draw opacity=0][line width=2.25]  (202.48,46.5) .. controls (205.96,36.88) and (215.18,30) .. (226,30) .. controls (234.72,30) and (242.4,34.47) .. (246.88,41.24) -- (226,55) -- cycle ; \draw  [color={rgb, 255:red, 223; green, 83; blue, 107 }  ,draw opacity=1 ][line width=2.25]  (202.48,46.5) .. controls (205.96,36.88) and (215.18,30) .. (226,30) .. controls (234.72,30) and (242.4,34.47) .. (246.88,41.24) ;  
\draw [color={rgb, 255:red, 223; green, 83; blue, 107 }  ,draw opacity=1 ][line width=2.25]  [dash pattern={on 6.75pt off 4.5pt}]  (34.1,19.91) .. controls (70.25,0.66) and (237,0.91) .. (278.1,20) ;
\draw [color={rgb, 255:red, 223; green, 83; blue, 107 }  ,draw opacity=1 ][line width=2.25]  [dash pattern={on 6.75pt off 4.5pt}]  (34.5,90.8) .. controls (70.65,110.05) and (237.4,109.8) .. (278.5,90.71) ;
\draw  [draw opacity=0][line width=2.25]  (246.8,68.87) .. controls (242.32,75.58) and (234.68,80) .. (226,80) .. controls (215.28,80) and (206.14,73.26) .. (202.59,63.79) -- (226,55) -- cycle ; \draw  [color={rgb, 255:red, 223; green, 83; blue, 107 }  ,draw opacity=1 ][line width=2.25]  (246.8,68.87) .. controls (242.32,75.58) and (234.68,80) .. (226,80) .. controls (215.28,80) and (206.14,73.26) .. (202.59,63.79) ;  
\draw [color={rgb, 255:red, 223; green, 83; blue, 107 }  ,draw opacity=1 ][line width=2.25]    (278.1,20) -- (246,40.16) ;
\draw [color={rgb, 255:red, 223; green, 83; blue, 107 }  ,draw opacity=1 ][line width=2.25]  [dash pattern={on 6.75pt off 4.5pt}]  (33.3,90.8) .. controls (69.45,110.05) and (236.2,109.8) .. (277.3,90.71) ;

\draw (7,42.4) node [anchor=north west][inner sep=0.75pt]  [font=\large,color={rgb, 255:red, 223; green, 83; blue, 107 }  ,opacity=1 ]  {$\textcolor[rgb]{0.87,0.33,0.42}{n}$};

\end{tikzpicture} & \tikzset{every picture/.style={line width=0.5pt}} 

\begin{tikzpicture}[x=0.55pt,y=0.55pt,yscale=-1,xscale=1]

\draw  [fill={rgb, 255:red, 0; green, 0; blue, 0 }  ,fill opacity=1 ] (61,56) .. controls (61,42.19) and (72.19,31) .. (86,31) .. controls (99.81,31) and (111,42.19) .. (111,56) .. controls (111,69.81) and (99.81,81) .. (86,81) .. controls (72.19,81) and (61,69.81) .. (61,56) -- cycle ;
\draw  [fill={rgb, 255:red, 0; green, 0; blue, 0 }  ,fill opacity=1 ] (201,55) .. controls (201,41.19) and (212.19,30) .. (226,30) .. controls (239.81,30) and (251,41.19) .. (251,55) .. controls (251,68.81) and (239.81,80) .. (226,80) .. controls (212.19,80) and (201,68.81) .. (201,55) -- cycle ;
\draw [color={rgb, 255:red, 223; green, 83; blue, 107 }  ,draw opacity=1 ][line width=2.25]    (34.1,19.91) -- (66.5,40.31) ;
\draw [color={rgb, 255:red, 223; green, 83; blue, 107 }  ,draw opacity=1 ][line width=2.25]    (246.1,70.31) -- (278.5,90.71) ;
\draw [color={rgb, 255:red, 223; green, 83; blue, 107 }  ,draw opacity=1 ][line width=2.25]    (66.9,70.4) -- (34.5,90.8) ;
\draw    (61.82,61.43) -- (29.42,81.83) ;
\draw    (283.07,29.6) -- (250.67,50) ;
\draw    (250.1,61.01) -- (282.5,81.41) ;
\draw    (29.1,28.93) -- (61.5,49.33) ;
\draw  [draw opacity=0][line width=2.25]  (65.54,41.62) .. controls (70.07,35.2) and (77.54,31) .. (86,31) .. controls (99.81,31) and (111,42.19) .. (111,56) .. controls (111,69.81) and (99.81,81) .. (86,81) .. controls (77.05,81) and (69.2,76.3) .. (64.79,69.23) -- (86,56) -- cycle ; \draw  [color={rgb, 255:red, 223; green, 83; blue, 107 }  ,draw opacity=1 ][line width=2.25]  (65.54,41.62) .. controls (70.07,35.2) and (77.54,31) .. (86,31) .. controls (99.81,31) and (111,42.19) .. (111,56) .. controls (111,69.81) and (99.81,81) .. (86,81) .. controls (77.05,81) and (69.2,76.3) .. (64.79,69.23) ;  
\draw  [draw opacity=0][line width=2.25]  (246.52,69.29) .. controls (242,75.76) and (234.49,80) .. (226,80) .. controls (212.19,80) and (201,68.81) .. (201,55) .. controls (201,41.19) and (212.19,30) .. (226,30) .. controls (234.72,30) and (242.4,34.47) .. (246.88,41.24) -- (226,55) -- cycle ; \draw  [color={rgb, 255:red, 223; green, 83; blue, 107 }  ,draw opacity=1 ][line width=2.25]  (246.52,69.29) .. controls (242,75.76) and (234.49,80) .. (226,80) .. controls (212.19,80) and (201,68.81) .. (201,55) .. controls (201,41.19) and (212.19,30) .. (226,30) .. controls (234.72,30) and (242.4,34.47) .. (246.88,41.24) ;  
\draw [color={rgb, 255:red, 223; green, 83; blue, 107 }  ,draw opacity=1 ][line width=2.25]  [dash pattern={on 6.75pt off 4.5pt}]  (34.1,19.91) .. controls (70.25,0.66) and (237,0.91) .. (278.1,20) ;
\draw [color={rgb, 255:red, 223; green, 83; blue, 107 }  ,draw opacity=1 ][line width=2.25]  [dash pattern={on 6.75pt off 4.5pt}]  (34.5,90.8) .. controls (70.65,110.05) and (237.4,109.8) .. (278.5,90.71) ;
\draw [color={rgb, 255:red, 223; green, 83; blue, 107 }  ,draw opacity=1 ][line width=2.25]    (278.1,20) -- (246,40.16) ;
\draw [color={rgb, 255:red, 223; green, 83; blue, 107 }  ,draw opacity=1 ][line width=2.25]  [dash pattern={on 6.75pt off 4.5pt}]  (34.5,90.8) .. controls (70.65,110.05) and (237.4,109.8) .. (278.5,90.71) ;

\draw (132.3,23.4) node [anchor=north west][inner sep=0.75pt]  [font=\large,color={rgb, 255:red, 223; green, 83; blue, 107 }  ,opacity=1 ]  {$\textcolor[rgb]{0.87,0.33,0.42}{n+1}$};

\end{tikzpicture} \\
        \hline 
        \tikzset{every picture/.style={line width=0.5pt}} 

\begin{tikzpicture}[x=0.55pt,y=0.55pt,yscale=-1,xscale=1]

\draw  [fill={rgb, 255:red, 0; green, 0; blue, 0 }  ,fill opacity=1 ] (61,56) .. controls (61,42.19) and (72.19,31) .. (86,31) .. controls (99.81,31) and (111,42.19) .. (111,56) .. controls (111,69.81) and (99.81,81) .. (86,81) .. controls (72.19,81) and (61,69.81) .. (61,56) -- cycle ;
\draw  [fill={rgb, 255:red, 0; green, 0; blue, 0 }  ,fill opacity=1 ] (201,55) .. controls (201,41.19) and (212.19,30) .. (226,30) .. controls (239.81,30) and (251,41.19) .. (251,55) .. controls (251,68.81) and (239.81,80) .. (226,80) .. controls (212.19,80) and (201,68.81) .. (201,55) -- cycle ;
\draw [color={rgb, 255:red, 223; green, 83; blue, 107 }  ,draw opacity=1 ][line width=2.25]    (107.9,44.91) -- (202.9,44.91) ;
\draw [color={rgb, 255:red, 40; green, 226; blue, 229 }  ,draw opacity=1 ][line width=2.25]    (108.7,65.71) -- (203.7,65.71) ;
\draw [color={rgb, 255:red, 223; green, 83; blue, 107 }  ,draw opacity=1 ][line width=2.25]    (34.1,19.91) -- (66.5,40.31) ;
\draw [color={rgb, 255:red, 40; green, 226; blue, 229 }  ,draw opacity=1 ][line width=2.25]    (246.1,70.31) -- (278.5,90.71) ;
\draw [color={rgb, 255:red, 40; green, 226; blue, 229 }  ,draw opacity=1 ][line width=2.25]    (65.7,70.4) -- (33.3,90.8) ;
\draw    (61.82,61.43) -- (29.42,81.83) ;
\draw    (283.07,29.6) -- (250.67,50) ;
\draw    (250.1,61.01) -- (282.5,81.41) ;
\draw    (29.1,28.93) -- (61.5,49.33) ;
\draw  [draw opacity=0][line width=2.25]  (109.57,64.35) .. controls (106.13,74.05) and (96.88,81) .. (86,81) .. controls (77.08,81) and (69.25,76.33) .. (64.83,69.3) -- (86,56) -- cycle ; \draw  [color={rgb, 255:red, 40; green, 226; blue, 229 }  ,draw opacity=1 ][line width=2.25]  (109.57,64.35) .. controls (106.13,74.05) and (96.88,81) .. (86,81) .. controls (77.08,81) and (69.25,76.33) .. (64.83,69.3) ;  
\draw  [draw opacity=0][line width=2.25]  (65.54,41.62) .. controls (70.07,35.2) and (77.54,31) .. (86,31) .. controls (96.27,31) and (105.09,37.19) .. (108.94,46.05) -- (86,56) -- cycle ; \draw  [color={rgb, 255:red, 223; green, 83; blue, 107 }  ,draw opacity=1 ][line width=2.25]  (65.54,41.62) .. controls (70.07,35.2) and (77.54,31) .. (86,31) .. controls (96.27,31) and (105.09,37.19) .. (108.94,46.05) ;  
\draw  [draw opacity=0][line width=2.25]  (202.48,46.5) .. controls (205.96,36.88) and (215.18,30) .. (226,30) .. controls (234.72,30) and (242.4,34.47) .. (246.88,41.24) -- (226,55) -- cycle ; \draw  [color={rgb, 255:red, 223; green, 83; blue, 107 }  ,draw opacity=1 ][line width=2.25]  (202.48,46.5) .. controls (205.96,36.88) and (215.18,30) .. (226,30) .. controls (234.72,30) and (242.4,34.47) .. (246.88,41.24) ;  
\draw [color={rgb, 255:red, 223; green, 83; blue, 107 }  ,draw opacity=1 ][line width=2.25]  [dash pattern={on 6.75pt off 4.5pt}]  (34.1,19.91) .. controls (70.25,0.66) and (237,0.91) .. (278.1,20) ;
\draw [color={rgb, 255:red, 40; green, 226; blue, 229 }  ,draw opacity=1 ][line width=2.25]  [dash pattern={on 6.75pt off 4.5pt}]  (34.5,90.8) .. controls (70.65,110.05) and (237.4,109.8) .. (278.5,90.71) ;
\draw  [draw opacity=0][line width=2.25]  (246.8,68.87) .. controls (242.32,75.58) and (234.68,80) .. (226,80) .. controls (215.28,80) and (206.14,73.26) .. (202.59,63.79) -- (226,55) -- cycle ; \draw  [color={rgb, 255:red, 40; green, 226; blue, 229 }  ,draw opacity=1 ][line width=2.25]  (246.8,68.87) .. controls (242.32,75.58) and (234.68,80) .. (226,80) .. controls (215.28,80) and (206.14,73.26) .. (202.59,63.79) ;  
\draw [color={rgb, 255:red, 223; green, 83; blue, 107 }  ,draw opacity=1 ][line width=2.25]    (278.1,20) -- (246,40.16) ;
\draw [color={rgb, 255:red, 40; green, 226; blue, 229 }  ,draw opacity=1 ][line width=2.25]  [dash pattern={on 6.75pt off 4.5pt}]  (33.3,90.8) .. controls (69.45,110.05) and (236.2,109.8) .. (277.3,90.71) ;

\draw (4,10.4) node [anchor=north west][inner sep=0.75pt]  [font=\large,color={rgb, 255:red, 223; green, 83; blue, 107 }  ,opacity=1 ]  {$\textcolor[rgb]{0.87,0.33,0.42}{n}$};
\draw (4,75.4) node [anchor=north west][inner sep=0.75pt]  [font=\large,color={rgb, 255:red, 40; green, 226; blue, 229 }  ,opacity=1 ]  {$\textcolor[rgb]{0.16,0.89,0.9}{m}$};

\end{tikzpicture} & \tikzset{every picture/.style={line width=0.5pt}} 

\begin{tikzpicture}[x=0.55pt,y=0.55pt,yscale=-1,xscale=1]

\draw  [fill={rgb, 255:red, 0; green, 0; blue, 0 }  ,fill opacity=1 ] (61,56) .. controls (61,42.19) and (72.19,31) .. (86,31) .. controls (99.81,31) and (111,42.19) .. (111,56) .. controls (111,69.81) and (99.81,81) .. (86,81) .. controls (72.19,81) and (61,69.81) .. (61,56) -- cycle ;
\draw  [fill={rgb, 255:red, 0; green, 0; blue, 0 }  ,fill opacity=1 ] (201,55) .. controls (201,41.19) and (212.19,30) .. (226,30) .. controls (239.81,30) and (251,41.19) .. (251,55) .. controls (251,68.81) and (239.81,80) .. (226,80) .. controls (212.19,80) and (201,68.81) .. (201,55) -- cycle ;
\draw [color={rgb, 255:red, 205; green, 11; blue, 188 }  ,draw opacity=1 ][line width=2.25]    (34.1,19.91) -- (66.5,40.31) ;
\draw [color={rgb, 255:red, 205; green, 11; blue, 188 }  ,draw opacity=1 ][line width=2.25]    (246.1,70.31) -- (278.5,90.71) ;
\draw [color={rgb, 255:red, 205; green, 11; blue, 188 }  ,draw opacity=1 ][line width=2.25]    (65.7,70.4) -- (33.3,90.8) ;
\draw    (61.82,61.43) -- (29.42,81.83) ;
\draw    (283.07,29.6) -- (250.67,50) ;
\draw    (250.1,61.01) -- (282.5,81.41) ;
\draw    (29.1,28.93) -- (61.5,49.33) ;
\draw  [draw opacity=0][line width=2.25]  (65.54,41.62) .. controls (70.07,35.2) and (77.54,31) .. (86,31) .. controls (99.81,31) and (111,42.19) .. (111,56) .. controls (111,69.81) and (99.81,81) .. (86,81) .. controls (77.05,81) and (69.2,76.3) .. (64.79,69.23) -- (86,56) -- cycle ; \draw  [color={rgb, 255:red, 205; green, 11; blue, 188 }  ,draw opacity=1 ][line width=2.25]  (65.54,41.62) .. controls (70.07,35.2) and (77.54,31) .. (86,31) .. controls (99.81,31) and (111,42.19) .. (111,56) .. controls (111,69.81) and (99.81,81) .. (86,81) .. controls (77.05,81) and (69.2,76.3) .. (64.79,69.23) ;  
\draw  [draw opacity=0][line width=2.25]  (246.52,69.29) .. controls (242,75.76) and (234.49,80) .. (226,80) .. controls (212.19,80) and (201,68.81) .. (201,55) .. controls (201,41.19) and (212.19,30) .. (226,30) .. controls (234.72,30) and (242.4,34.47) .. (246.88,41.24) -- (226,55) -- cycle ; \draw  [color={rgb, 255:red, 205; green, 11; blue, 188 }  ,draw opacity=1 ][line width=2.25]  (246.52,69.29) .. controls (242,75.76) and (234.49,80) .. (226,80) .. controls (212.19,80) and (201,68.81) .. (201,55) .. controls (201,41.19) and (212.19,30) .. (226,30) .. controls (234.72,30) and (242.4,34.47) .. (246.88,41.24) ;  
\draw [color={rgb, 255:red, 205; green, 11; blue, 188 }  ,draw opacity=1 ][line width=2.25]    (278.1,20) -- (246,40.16) ;
\draw [color={rgb, 255:red, 205; green, 11; blue, 188 }  ,draw opacity=1 ][line width=2.25]  [dash pattern={on 6.75pt off 4.5pt}]  (34.1,19.91) .. controls (70.25,0.66) and (237,0.91) .. (278.1,20) ;
\draw [color={rgb, 255:red, 205; green, 11; blue, 188 }  ,draw opacity=1 ][line width=2.25]  [dash pattern={on 6.75pt off 4.5pt}]  (33.3,90.8) .. controls (69.45,110.05) and (236.2,109.8) .. (277.3,90.71) ;

\draw (130,25.4) node [anchor=north west][inner sep=0.75pt]  [font=\large,color={rgb, 255:red, 205; green, 11; blue, 188 }  ,opacity=1 ]  {$\textcolor[rgb]{0.8,0.04,0.74}{n+m}$};

\end{tikzpicture} \\
        \hline
    \end{tabular}
    \caption{Deleting an edge in a packaged ribbon graph.}
    \label{Del}
\end{table}

\begin{definition}
    Let $\pG=(\bG,\V,\wV,\F,\wF)$ be a packaged ribbon graph, $V$ be its vertex set, and $e$ be an edge in $\bG$. We use $\pG/e$ to denote the packaged ribbon graph $(\bG/e,\V',\wV',\F',\wF')$ obtained by \emph{contracting} $e$ and define it as follows. There is a natural correspondence between the boundary components of $\bG$ and those of $\bG/e$. This correspondence induces a partition $\F'$ on the boundary components of $\bG/e$ and a weighting $\wF'$. For the vertices, $\V'$ and $\wV'$ are defined as follows.
    \begin{enumerate}
        \item\label{ccase1} If $e$ is a loop incident to $u\in V$,         
        then contract $e$ in $\bG$ and let $u'$, $v'$ denote the two vertices formed by the contraction. Obtain the partition $\V'$ from $\V$ by replacing $[u]$ with $[u']:=[u]\cup \{u',v'\}\setminus \{u\}$. The other blocks are unchanged. Define
        \[\wV'([x]):=
        \begin{cases}
        \wV([x])+1 & \text{ if } [x]=[u'],\\
        \wV([x]) & \text{ if } [x]\neq [u'].
        \end{cases}\]
        
        \item\label{ccase2} If $e$ is a non-loop edge incident to $u,v\in V$ and $[u]=[v]$, then contract $e$ in $\bG$ and let $u'$ denote the vertex formed by the contraction. Obtain the partition $\V'$ from $\V$ by replacing $[u]$ with $[u']:=[u]\cup \{u'\}\setminus \{u,v\}$. The other blocks are unchanged. Define
        \[\wV'([x]):=
        \begin{cases}
        \wV([x])+1 & \text{ if } [x]=[u'],\\
        \wV([x]) & \text{ if } [x]\neq [u'].
        \end{cases}\]
        
        \item\label{ccase3} If $e$ is a non-loop edge incident to $u,v\in V$ and $[u]\neq [v]$, then contract $e$ in $\bG$ and let $u'$ denote the vertex formed by the contraction. Obtain the partition $\V'$ from $\V$ by removing $[u]$ and $[v]$, and inserting $[u']:=[u]\cup [v]\cup\{u'\}\setminus \{u,v\}$. The other blocks are unchanged. Define
        \[\wV'([x]):=
        \begin{cases}
        \wV([u])+\wV([v]) & \text{ if } [x]=[u'],\\
        \wV([x]) & \text{ if } [x]\neq [u'].
        \end{cases}\]
    \end{enumerate}
    
\end{definition}

\begin{table}
    \centering
    \begin{tabular}{| c | c |}
        \hline
        $\pG$ & $\pG/e$ \\
        \hline
          \tikzset{every picture/.style={line width=0.5pt}} 

\begin{tikzpicture}[x=0.55pt,y=0.55pt,yscale=-1,xscale=1]

\draw  [draw opacity=0] (144.6,80.47) .. controls (151.46,95.38) and (164.29,105.41) .. (178.98,105.41) .. controls (200.89,105.41) and (218.65,83.09) .. (218.65,55.56) .. controls (218.65,28.03) and (200.89,5.72) .. (178.98,5.72) .. controls (164.79,5.72) and (152.35,15.06) .. (145.34,29.12) -- (178.98,55.56) -- cycle ; \draw   (144.6,80.47) .. controls (151.46,95.38) and (164.29,105.41) .. (178.98,105.41) .. controls (200.89,105.41) and (218.65,83.09) .. (218.65,55.56) .. controls (218.65,28.03) and (200.89,5.72) .. (178.98,5.72) .. controls (164.79,5.72) and (152.35,15.06) .. (145.34,29.12) ;  
\draw  [draw opacity=0] (157.07,80.87) .. controls (162.47,88.52) and (170.28,93.32) .. (178.98,93.32) .. controls (195.28,93.32) and (208.5,76.42) .. (208.5,55.56) .. controls (208.5,34.71) and (195.28,17.81) .. (178.98,17.81) .. controls (170.66,17.81) and (163.15,22.2) .. (157.79,29.27) -- (178.98,55.56) -- cycle ; \draw   (157.07,80.87) .. controls (162.47,88.52) and (170.28,93.32) .. (178.98,93.32) .. controls (195.28,93.32) and (208.5,76.42) .. (208.5,55.56) .. controls (208.5,34.71) and (195.28,17.81) .. (178.98,17.81) .. controls (170.66,17.81) and (163.15,22.2) .. (157.79,29.27) ;  
\draw    (194.38,23.79) -- (171.3,38) ;
\draw    (172.35,70.66) -- (196.63,85.54) ;
\draw  [fill={rgb, 255:red, 223; green, 83; blue, 107 }  ,fill opacity=1 ] (124.7,55.06) .. controls (124.7,40.51) and (136.5,28.71) .. (151.05,28.71) .. controls (165.6,28.71) and (177.4,40.51) .. (177.4,55.06) .. controls (177.4,69.62) and (165.6,81.41) .. (151.05,81.41) .. controls (136.5,81.41) and (124.7,69.62) .. (124.7,55.06) -- cycle ;
\draw    (97.85,18.78) -- (130.25,39.18) ;
\draw    (129.45,69.26) -- (97.05,89.66) ;
\draw    (125.48,59.8) -- (93.08,80.2) ;
\draw    (92.77,28.13) -- (125.17,48.53) ;
\draw    (201.5,30.91) -- (176.05,47) ;
\draw    (176.35,62.63) -- (202,78.16) ;

\draw (124,82.4) node [anchor=north west][inner sep=0.75pt]  [font=\large,color={rgb, 255:red, 223; green, 83; blue, 107 }  ,opacity=1 ]  {$\textcolor[rgb]{0.87,0.33,0.42}{n}$};

\end{tikzpicture} & \tikzset{every picture/.style={line width=0.5pt}} 

\begin{tikzpicture}[x=0.55pt,y=0.55pt,yscale=-1,xscale=1]

\draw  [fill={rgb, 255:red, 223; green, 83; blue, 107 }  ,fill opacity=1 ] (61,56) .. controls (61,42.19) and (72.19,31) .. (86,31) .. controls (99.81,31) and (111,42.19) .. (111,56) .. controls (111,69.81) and (99.81,81) .. (86,81) .. controls (72.19,81) and (61,69.81) .. (61,56) -- cycle ;
\draw  [fill={rgb, 255:red, 223; green, 83; blue, 107 }  ,fill opacity=1 ] (201,55) .. controls (201,41.19) and (212.19,30) .. (226,30) .. controls (239.81,30) and (251,41.19) .. (251,55) .. controls (251,68.81) and (239.81,80) .. (226,80) .. controls (212.19,80) and (201,68.81) .. (201,55) -- cycle ;
\draw    (34.1,19.91) -- (66.5,40.31) ;
\draw    (246.1,70.31) -- (278.5,90.71) ;
\draw    (65.7,70.4) -- (33.3,90.8) ;
\draw    (278.1,20) -- (245.7,40.4) ;
\draw    (62.07,60.93) -- (29.67,81.33) ;
\draw    (283.07,29.6) -- (250.67,50) ;
\draw    (250.1,61.01) -- (282.5,81.41) ;
\draw    (29.1,28.93) -- (61.5,49.33) ;

\draw (137,81.4) node [anchor=north west][inner sep=0.75pt]  [font=\large,color={rgb, 255:red, 223; green, 83; blue, 107 }  ,opacity=1 ]  {$\textcolor[rgb]{0.87,0.33,0.42}{n+1}$};

\end{tikzpicture} \\
        \hline
        \tikzset{every picture/.style={line width=0.5pt}} 

\begin{tikzpicture}[x=0.55pt,y=0.55pt,yscale=-1,xscale=1]

\draw  [fill={rgb, 255:red, 223; green, 83; blue, 107 }  ,fill opacity=1 ] (61,56) .. controls (61,42.19) and (72.19,31) .. (86,31) .. controls (99.81,31) and (111,42.19) .. (111,56) .. controls (111,69.81) and (99.81,81) .. (86,81) .. controls (72.19,81) and (61,69.81) .. (61,56) -- cycle ;
\draw  [fill={rgb, 255:red, 223; green, 83; blue, 107 }  ,fill opacity=1 ] (201,55) .. controls (201,41.19) and (212.19,30) .. (226,30) .. controls (239.81,30) and (251,41.19) .. (251,55) .. controls (251,68.81) and (239.81,80) .. (226,80) .. controls (212.19,80) and (201,68.81) .. (201,55) -- cycle ;
\draw    (107.9,44.91) -- (202.9,44.91) ;
\draw    (108.7,65.71) -- (203.7,65.71) ;
\draw    (34.1,19.91) -- (66.5,40.31) ;
\draw    (246.1,70.31) -- (278.5,90.71) ;
\draw    (65.7,70.4) -- (33.3,90.8) ;
\draw    (278.1,20) -- (245.7,40.4) ;
\draw    (61.82,61.43) -- (29.42,81.83) ;
\draw    (283.07,29.6) -- (250.67,50) ;
\draw    (250.1,61.01) -- (282.5,81.41) ;
\draw    (29.1,28.93) -- (61.5,49.33) ;

\draw (149,81.4) node [anchor=north west][inner sep=0.75pt]  [font=\large,color={rgb, 255:red, 223; green, 83; blue, 107 }  ,opacity=1 ]  {$\textcolor[rgb]{0.87,0.33,0.42}{n}$};

\end{tikzpicture} & \tikzset{every picture/.style={line width=0.5pt}} 

\begin{tikzpicture}[x=0.55pt,y=0.55pt,yscale=-1,xscale=1]

\draw  [fill={rgb, 255:red, 223; green, 83; blue, 107 }  ,fill opacity=1 ] (131.7,54.06) .. controls (131.7,39.51) and (143.5,27.71) .. (158.05,27.71) .. controls (172.6,27.71) and (184.4,39.51) .. (184.4,54.06) .. controls (184.4,68.62) and (172.6,80.41) .. (158.05,80.41) .. controls (143.5,80.41) and (131.7,68.62) .. (131.7,54.06) -- cycle ;
\draw    (104.6,18.36) -- (137,38.76) ;
\draw    (136.2,68.85) -- (103.8,89.25) ;
\draw    (132.23,59.38) -- (99.83,79.78) ;
\draw    (99.52,27.72) -- (131.92,48.12) ;
\draw    (179.85,68.85) -- (212.25,89.25) ;
\draw    (211.85,18.54) -- (179.45,38.94) ;
\draw    (216.15,27.8) -- (183.75,48.2) ;
\draw    (184.18,59.8) -- (216.58,80.2) ;

\draw (138,83.4) node [anchor=north west][inner sep=0.75pt]  [font=\large,color={rgb, 255:red, 223; green, 83; blue, 107 }  ,opacity=1 ]  {$\textcolor[rgb]{0.87,0.33,0.42}{n+1}$};

\end{tikzpicture} \\
        \hline
        \tikzset{every picture/.style={line width=0.5pt}} 

\begin{tikzpicture}[x=0.55pt,y=0.55pt,yscale=-1,xscale=1]

\draw  [fill={rgb, 255:red, 223; green, 83; blue, 107 }  ,fill opacity=1 ] (61,56) .. controls (61,42.19) and (72.19,31) .. (86,31) .. controls (99.81,31) and (111,42.19) .. (111,56) .. controls (111,69.81) and (99.81,81) .. (86,81) .. controls (72.19,81) and (61,69.81) .. (61,56) -- cycle ;
\draw  [fill={rgb, 255:red, 40; green, 226; blue, 229 }  ,fill opacity=1 ] (201,55) .. controls (201,41.19) and (212.19,30) .. (226,30) .. controls (239.81,30) and (251,41.19) .. (251,55) .. controls (251,68.81) and (239.81,80) .. (226,80) .. controls (212.19,80) and (201,68.81) .. (201,55) -- cycle ;
\draw    (107.9,44.91) -- (202.9,44.91) ;
\draw    (108.7,65.71) -- (203.7,65.71) ;
\draw    (34.1,19.91) -- (66.5,40.31) ;
\draw    (246.1,70.31) -- (278.5,90.71) ;
\draw    (65.7,70.4) -- (33.3,90.8) ;
\draw    (278.1,20) -- (245.7,40.4) ;
\draw    (282.4,29.27) -- (250,49.67) ;
\draw    (61.73,60.93) -- (29.33,81.33) ;
\draw    (250.43,61.27) -- (282.83,81.67) ;
\draw    (29.02,29.27) -- (61.42,49.67) ;

\draw (80,83.4) node [anchor=north west][inner sep=0.75pt]  [font=\large,color={rgb, 255:red, 223; green, 83; blue, 107 }  ,opacity=1 ]  {$\textcolor[rgb]{0.87,0.33,0.42}{{\displaystyle n}}$};
\draw (219,83.4) node [anchor=north west][inner sep=0.75pt]  [font=\large,color={rgb, 255:red, 40; green, 226; blue, 229 }  ,opacity=1 ]  {$\textcolor[rgb]{0.16,0.89,0.9}{m}$};

\end{tikzpicture} & \tikzset{every picture/.style={line width=0.5pt}} 

\begin{tikzpicture}[x=0.55pt,y=0.55pt,yscale=-1,xscale=1]

\draw  [fill={rgb, 255:red, 205; green, 11; blue, 188 }  ,fill opacity=1 ] (129.7,55.06) .. controls (129.7,40.51) and (141.5,28.71) .. (156.05,28.71) .. controls (170.6,28.71) and (182.4,40.51) .. (182.4,55.06) .. controls (182.4,69.62) and (170.6,81.41) .. (156.05,81.41) .. controls (141.5,81.41) and (129.7,69.62) .. (129.7,55.06) -- cycle ;
\draw    (102.6,19.36) -- (135,39.76) ;
\draw    (134.2,69.85) -- (101.8,90.25) ;
\draw    (130.23,60.38) -- (97.83,80.78) ;
\draw    (97.52,28.72) -- (129.92,49.12) ;
\draw    (177.85,69.85) -- (210.25,90.25) ;
\draw    (209.85,19.54) -- (177.45,39.94) ;
\draw    (214.15,28.8) -- (181.75,49.2) ;
\draw    (182.18,60.8) -- (214.58,81.2) ;

\draw (133,82.4) node [anchor=north west][inner sep=0.75pt]  [font=\large,color={rgb, 255:red, 205; green, 11; blue, 188 }  ,opacity=1 ]  {$\textcolor[rgb]{0.8,0.04,0.74}{n+m}$};

\end{tikzpicture} \\
        \hline
    \end{tabular}
    \caption{Contracting an edge in a packaged ribbon graph.}
    \label{Con}
\end{table}
See Figure~\ref{PolyDC} for an example of deletion and contraction in a packaged ribbon graph. In the figure, deletion branches left and contraction right. The operations are applied in the order $e, f, g$. 

The reader may note a similarity with the deletion and contraction operations of Noble and Welsh's U-polynomial~\cite{zbMATH01314959} for vertex weighted graphs however it is important to distinguish the different behavior of Noble and Welsh's contraction on loops and the contraction used here.

To more succinctly distinguish the edge types in the definitions of deletion and contraction, we introduce the following notation.
\begin{definition}
 For a packaged ribbon graph $\pG=(\bG,\V,\wV,\F,\wF)$ and an edge $e\in E$,
 \begin{enumerate}
\item   let $\n$ denote the number of blocks in $\F$ that contain a boundary component that intersects $e$; and
\item let $\m$ denote the number of blocks in $\V$ that contain a vertex incident to $e$. 
  \end{enumerate}
  When $\pG$ is not clear from context we shall shall write $\eta_{\pG}(e)$ and $\mu_{\pG}(e)$. 
 \end{definition}
 In this notation,  Cases~\ref{dcase1} and~\ref{dcase2} in the definition of deletion have $\n=1$, while Case~\ref{dcase3} has $\n=2$. 
Similarly, Cases~\ref{ccase1} and~\ref{ccase2}   in the definition of contraction have $\m=1$, while Case~\ref{ccase3} has $\m=2$.

\medskip

The \emph{dual} $\pG^*=(\bG^*,\V^*,\wVdual,\F^*,\wFdual)$ of a packaged ribbon graph $\pG=(\bG,\V,\wV,\F,\wF)$ is obtained in the following way. The ribbon graph $\bG^*$ is the dual of the ribbon graph $\bG$. As the vertices of $\bG$ are in one-to-one correspondence with the boundary components of $\bG^*$, the vertex partition $\V$ on $V(\bG)$ induces a boundary component partition $\F^*$ on the set of all boundary components in $\bG^*$. In a similar way, we can obtain a vertex partition $\V^*$ on $V(\bG^*)$ from the boundary component partition $\F$ on the set of all boundary components in $\bG$. The partition $\V^*$ (resp. $\F^*$) naturally inherits the weighting $\wVdual$ (resp. $\wFdual$) from the weighting $\wF$ (resp. $\wV$) on $\F$ (resp. $\V$). 
 At times we will abuse notation and simply denote the dual of $\pG=(\bG,\V,\wV,\F,\wF)$ by $\pG^*=(\bG^*,\F,\wF,\V,\wV)$.

The following proposition extends the classical and well-known identity  $G^*/e = (G\ba e)^*$ for a plane graph $G$ and its dual $G^*$.
\begin{proposition}\label{dualitydelcon}
    Let $\pG=(\bG,\V,\wV,\F,\wF)$ be a packaged ribbon graph and $e\in E$ be an edge. Then
    \begin{align*}
	    (\pG\backslash e)^* = \pG^*/e \hspace{24pt} \text{and} \hspace{24pt} (\pG/ e)^* = \pG^*\backslash e.
	\end{align*}
\end{proposition}

\begin{proof}
    As duality is an involution, it suffices to only prove one of the theorem statements. We will show that $(\pG\backslash e)^* = \pG^*/e$. Denote the dual by $\pG^*=(\bG^*,\F,\wF,\V,\wV)$ and let $F$ be its set of boundary components. For the ribbon graph $\bG$, it is well established (e.g., see~\cite{graphsonsurfaces}) that $(\bG\ba e)^*=\bG^*/e$. We just need to check that the partitions and their weightings correctly align. 
    
    Observe that the three edge types considered in the definition of deletion within the primal correspond to the three edge types considered in the definition of contraction in the dual. That is to say, for example, that if an edge intersects one boundary component twice in the primal, then the corresponding edge in the dual will be a loop. Additionally, notice that the partitions and weightings are changed in the same manner for the same corresponding edge case of deletion and contraction. The only difference is that one affects the partition on the boundary components and the other affects the partition on the vertices. Hence, taking the dual after deleting an edge gives the same packaged ribbon graph as contracting the corresponding edge in the dual.
\end{proof}

\medskip

We will make use of the following graph that is naturally associated with a packaged ribbon graph. It arises from its underlying graph by identifying all vertices that are in the  same block in the partition. 
\begin{definition} 
    Let $\pG=(\bG,\V,\wV,\F,\wF)$ be a packaged ribbon graph. Its \emph{packaging} $G(\bG;\V)$ is the vertex-weighted graph obtained in the following way. Create a vertex for each block in the partition $\V$. 
    The vertex set of $G(\bG;\V)$ is  the partition $\V$.
    There is an edge $([u], [v])$ in $G(\bG;\V)$ for each edge $(u,v)$ in $\bG$, and each edge has this form. 
   The vertices of $G(\bG;\V)$ are weighted by $\wV$.
\end{definition}
 An example is shown in Figure \ref{Packaging}. 
We note that there is an asymmetry in our definitions as we do not define the packaging for the boundary component partition $\F$. We could do this, but for the sake of notational simplicity we work with  $G(\bG^*;\F)$ instead. 

Every edge of $G(\bG;\V)$ corresponds to a unique edge of $\bG$, and every vertex of $G(\bG;V)$ corresponds to a distinct block in $\V$. Thus, each subgraph $K$ of $G(\bG;\V)$ gives rise to a ribbon subgraph $\mathbb{K}=(U,A)$ where
$U:=\{u\in [w]\in \V : w\in V(K)\}$ 
(thus $U$ consists of all vertices in $\bG$ that are sent to vertices in $K$ when forming the packaging)
and $A$ is the set of edges in $\bG$ corresponding to the edges in $K$. We denote the ribbon subgraph $\mathbb{K}$ obtained in this way by $\bG[K]$.
Moreover, in this case we set 
\[f(\bG[K]):=f(\mathbb{K}).\]
Figure~\ref{Packaging} shows an example of $\bG[K]$ where $f(\bG[K])=4$.

\begin{figure}
\centering
    \centering
    \input{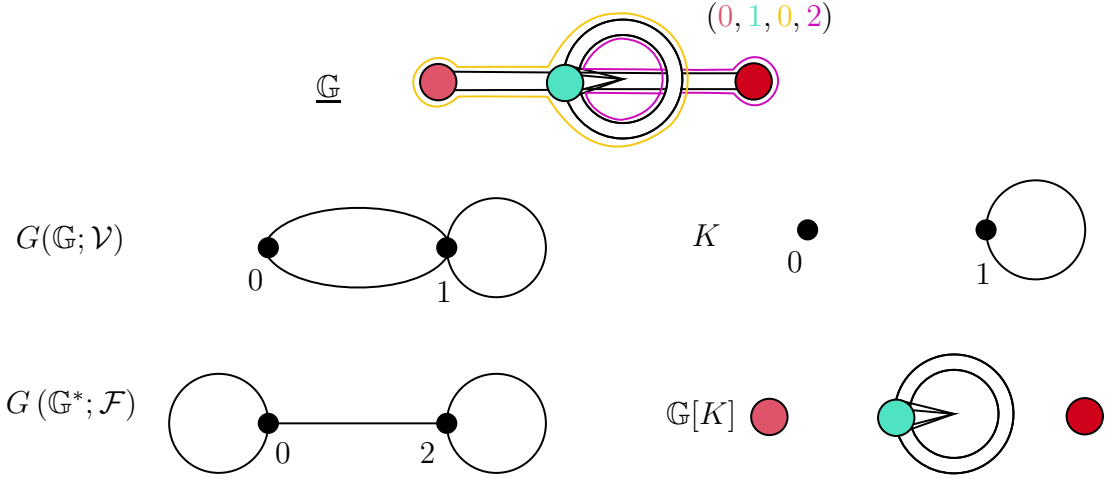}
    \caption{A packaged ribbon graph $\pG$, its packagings $G(\bG;\V)$ and $G(\bG^*;\F)$, and a subgraph $K$ of $G(\bG;\V)$ with its corresponding ribbon subgraph $\bG[K]$.}
    \label{Packaging}
\end{figure}

\subsection{A topological interpretation of packaged ribbon graphs}
Packaged ribbon graphs can be interpreted as graphs embedded in pseudo-surfaces that are equipped with vertex and face weights. Extending the work in~\cite{HM}, in this subsection we define these embedded graphs and describe how to move between the two types of object. 

We define a \emph{pseudo-surface} to be a topological space that results from a closed surface by contracting  finitely many closed paths to points. A \emph{pinch point} is any point in a pseudo-surface that does not have a neighbourhood homeomorphic to a disc. If a pseudo-surface contains no pinch points then it is also a surface.
(Although not as commonly studied as graphs in surfaces, graphs in pseudo-surfaces are frequently studied in topological graph theory. See, for example, \cite{zbMATH05064525,zbMATH06330594,zbMATH06473566, zbMATH01624430, HM} for a taste of the area.)

We say a graph $G=(V,E)$ is \emph{embedded in a pseudo-surface $\Sigma$} if $V$ is a set of points in the pseudo-surface, any pinch points are contained in $V$, and $E$ is a set of simple paths between the points such that two edges can only intersect at their incident vertices. Let $F$ denote the set of connected regions in $\Sigma\setminus G$. Define a \emph{vertex weighting} (resp. \emph{region weighting}) to be a mapping $\phi :V\to \mathbb{N}_0$ (resp. $\rho : F\to \mathbb{N}_0$).

Given a graph $G$ embedded in a pseudo-surface $\Sigma$ with vertex and region weights, $\phi$ and $\rho$, obtain its corresponding  packaged ribbon graph $\pG=(\bG,\V,\wV,\F,\wF)$ in the following way.
\begin{enumerate}
    \item Form a graph embedded in a surface by deleting a small neighbourhood of each pinch point, then contracting each boundary component to form a vertex.  Next form a ribbon graph  $\bG$ in the usual way by taking a neighbourhood of the resulting embedded graph.
     \item  For the partition  $\V$, place two vertices in the same block if and only if they arise from the same vertex of $G$ under this construction.  Denote the ribbon graph's set of boundary components by $F$. Define the partition $\F$ on $F$ such that two boundary components share a partition block if and only if they are in the same connected component of $\Sigma\setminus G$.   
    \item There is one-to-one correspondence between the vertices in $G$ and the blocks of $\V$, as well as between the connected regions of $\Sigma\setminus G$ and the blocks of $\F$. So the weightings $\phi$ and $\rho$ naturally induce weightings $\wV:\V\to \mathbb{N}_0$ and $\wF:\F\to\mathbb{N}_0$.
\end{enumerate}

On the other hand, given a packaged ribbon graph $\pG=(\bG,\V,\wV,\F,\wF)$ we can obtain an associated vertex and region weighted graph $G$ embedded in a pseudo-surface $\hat{\Sigma}$ in the following way.
\begin{enumerate}
    \item Cap off each boundary component of $\bG$ with a disc by identifying the boundary of the disc to the boundary component. Contract the vertex discs to points, and the edge discs to paths between these points. This results in the  graph cellularly embedded in a surface that corresponds to $\bG$ in the usual way.
    \item Next, for vertices that share a block in $\V$, identify them together into a pinch point. (If a block has size one, the vertex contained in it is unchanged). This gives us a graph $G$ embedded in a pseudo-surface $\Sigma$. For boundary components that share a block in $\F$, add a handle between the corresponding connected regions to connect them creating a new surface $\hat{\Sigma}$.
    \item Let $V$ denote the set of vertices in $G$ and $F$ denote the set of connected regions in $\hat{\Sigma}\setminus G$. There is a one-to-one correspondence between the vertices in $G$ and the blocks of $\V$, as well as between the connected regions of $\Sigma\setminus G$ and the blocks of $\F$. So the weightings $\wV$ and $\wF$ naturally induce the vertex weighting $\phi :V\to \mathbb{N}_0$ and region weighting $\rho : F\to \mathbb{N}_0$.
\end{enumerate}

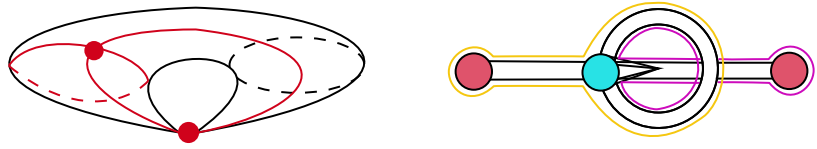
\begin{figure}
    \centering
    \tikzset{every picture/.style={line width=0.75pt}} 

\begin{tikzpicture}[x=0.75pt,y=0.75pt,yscale=-1,xscale=1]

\draw    (234.4,97.2) .. controls (306.4,46.2) and (165.4,49.2) .. (225.9,97.2) ;
\draw    (235.57,97.2) .. controls (368.4,73.2) and (324.4,37.2) .. (234.4,34.2) .. controls (134.4,37.2) and (89.4,74.2) .. (225.9,97.2) ;
\draw  [dash pattern={on 4.5pt off 4.5pt}] (251.4,63.2) .. controls (251.4,55.47) and (266.62,49.2) .. (285.4,49.2) .. controls (304.18,49.2) and (319.4,55.47) .. (319.4,63.2) .. controls (319.4,70.93) and (304.18,77.2) .. (285.4,77.2) .. controls (266.62,77.2) and (251.4,70.93) .. (251.4,63.2) -- cycle ;
\draw  [color={rgb, 255:red, 208; green, 2; blue, 27 }  ,draw opacity=1 ][fill={rgb, 255:red, 208; green, 2; blue, 27 }  ,fill opacity=1 ] (178.69,55.9) .. controls (178.69,53.49) and (180.65,51.53) .. (183.06,51.53) .. controls (185.47,51.53) and (187.42,53.49) .. (187.42,55.9) .. controls (187.42,58.31) and (185.47,60.26) .. (183.06,60.26) .. controls (180.65,60.26) and (178.69,58.31) .. (178.69,55.9) -- cycle ;
\draw [color={rgb, 255:red, 208; green, 2; blue, 27 }  ,draw opacity=1 ]   (140.07,63.53) .. controls (154.07,42.87) and (206.73,54.87) .. (210.4,71.2) ;
\draw [color={rgb, 255:red, 208; green, 2; blue, 27 }  ,draw opacity=1 ] [dash pattern={on 4.5pt off 4.5pt}]  (140.07,63.53) .. controls (170.73,88.87) and (198.73,83.53) .. (210.4,71.2) ;
\draw [color={rgb, 255:red, 208; green, 2; blue, 27 }  ,draw opacity=1 ]   (235.57,97.2) .. controls (292.4,86.2) and (317.4,55.2) .. (234.4,45.2) .. controls (199.62,45.2) and (183.37,51.14) .. (180.19,59.43) .. controls (175.88,70.68) and (195.67,86.26) .. (225.9,97.2) ;
\draw  [color={rgb, 255:red, 208; green, 2; blue, 27 }  ,draw opacity=1 ][fill={rgb, 255:red, 208; green, 2; blue, 27 }  ,fill opacity=1 ] (225.9,97.2) .. controls (225.9,94.53) and (228.06,92.37) .. (230.73,92.37) .. controls (233.4,92.37) and (235.57,94.53) .. (235.57,97.2) .. controls (235.57,99.87) and (233.4,102.03) .. (230.73,102.03) .. controls (228.06,102.03) and (225.9,99.87) .. (225.9,97.2) -- cycle ;
\draw    (382.06,70.52) -- (430.6,70.4) ;
\draw    (381.4,61.2) -- (431.8,61.6) ;
\draw    (490.4,62) -- (441,61.6) ;
\draw    (489.4,70) -- (440.6,70) ;
\draw  [draw opacity=0] (438.58,57.72) .. controls (441.89,44.26) and (454.25,34.48) .. (468.65,34.93) .. controls (485.21,35.45) and (498.22,49.3) .. (497.7,65.86) .. controls (497.18,82.42) and (483.33,95.42) .. (466.77,94.9) .. controls (454.06,94.5) and (443.44,86.25) .. (439.43,74.95) -- (467.71,64.92) -- cycle ; \draw   (438.58,57.72) .. controls (441.89,44.26) and (454.25,34.48) .. (468.65,34.93) .. controls (485.21,35.45) and (498.22,49.3) .. (497.7,65.86) .. controls (497.18,82.42) and (483.33,95.42) .. (466.77,94.9) .. controls (454.06,94.5) and (443.44,86.25) .. (439.43,74.95) ;  
\draw  [draw opacity=0] (445.34,62.94) .. controls (446.36,51.58) and (456.01,42.68) .. (467.75,42.7) .. controls (480.15,42.72) and (490.19,52.68) .. (490.17,64.96) .. controls (490.15,77.23) and (480.08,87.16) .. (467.67,87.14) .. controls (457.23,87.12) and (448.47,80.06) .. (445.97,70.5) -- (467.71,64.92) -- cycle ; \draw   (445.34,62.94) .. controls (446.36,51.58) and (456.01,42.68) .. (467.75,42.7) .. controls (480.15,42.72) and (490.19,52.68) .. (490.17,64.96) .. controls (490.15,77.23) and (480.08,87.16) .. (467.67,87.14) .. controls (457.23,87.12) and (448.47,80.06) .. (445.97,70.5) ;  
\draw  [color={rgb, 255:red, 0; green, 0; blue, 0 }  ,draw opacity=1 ][fill={rgb, 255:red, 223; green, 83; blue, 107 }  ,fill opacity=1 ][line width=0.75]  (365.43,66.46) .. controls (365.43,61.39) and (369.54,57.28) .. (374.6,57.28) .. controls (379.67,57.28) and (383.78,61.39) .. (383.78,66.46) .. controls (383.78,71.52) and (379.67,75.63) .. (374.6,75.63) .. controls (369.54,75.63) and (365.43,71.52) .. (365.43,66.46) -- cycle ;
\draw  [color={rgb, 255:red, 0; green, 0; blue, 0 }  ,draw opacity=1 ][fill={rgb, 255:red, 40; green, 226; blue, 229 }  ,fill opacity=1 ][line width=0.75]  (429.41,66.9) .. controls (429.41,61.83) and (433.51,57.72) .. (438.58,57.72) .. controls (443.65,57.72) and (447.76,61.83) .. (447.76,66.9) .. controls (447.76,71.97) and (443.65,76.07) .. (438.58,76.07) .. controls (433.51,76.07) and (429.41,71.97) .. (429.41,66.9) -- cycle ;
\draw    (525.4,62) -- (497.8,62) ;
\draw    (525.8,70) -- (496.6,70) ;
\draw  [color={rgb, 255:red, 0; green, 0; blue, 0 }  ,draw opacity=1 ][fill={rgb, 255:red, 223; green, 83; blue, 107 }  ,fill opacity=1 ][line width=0.75]  (524.51,66.12) .. controls (524.51,61.05) and (528.62,56.94) .. (533.69,56.94) .. controls (538.75,56.94) and (542.86,61.05) .. (542.86,66.12) .. controls (542.86,71.18) and (538.75,75.29) .. (533.69,75.29) .. controls (528.62,75.29) and (524.51,71.18) .. (524.51,66.12) -- cycle ;
\draw [color={rgb, 255:red, 205; green, 11; blue, 188 }  ,draw opacity=1 ][line width=0.75] [line join = round][line cap = round]   (488.67,60.17) .. controls (486.27,60.17) and (460.67,59.83) .. (447.67,59.83) .. controls (452,49.5) and (461.33,44.33) .. (467,44.5) .. controls (493.4,45.7) and (496,81.5) .. (467,85.5) .. controls (461.33,85.5) and (450.67,80.83) .. (449,71.83) .. controls (460,71.83) and (476.33,72.17) .. (488,72.17) ;
\draw [color={rgb, 255:red, 205; green, 11; blue, 188 }  ,draw opacity=1 ][line width=0.75] [line join = round][line cap = round]   (498.33,60.17) .. controls (508,60.5) and (511.33,60.17) .. (523.33,60.17) .. controls (533.67,47.17) and (546,56.83) .. (546,65.83) .. controls (546,74.83) and (534,84.83) .. (523.67,72.17) .. controls (510.33,71.83) and (510.67,71.83) .. (498,71.83) ;
\draw  [color={rgb, 255:red, 245; green, 199; blue, 16 }  ,draw opacity=1 ][line width=0.75] [line join = round][line cap = round] (470.67,31.83) .. controls (463.5,31.83) and (445.33,28.83) .. (430,58.83) .. controls (426.46,58.71) and (400.83,58.67) .. (384.33,58.83) .. controls (375.33,48.5) and (362,57.5) .. (362,66.5) .. controls (362,75.5) and (372.67,84.83) .. (384.67,73.83) .. controls (385.67,73.5) and (428.67,73.83) .. (429.67,73.83) .. controls (446,100.17) and (467,107.17) .. (491.33,88.5) .. controls (505.67,74.83) and (506,37.5) .. (470.67,31.83) -- cycle ;

\end{tikzpicture}
    \caption{A graph embedded in a pseudo-surface and its corresponding packaged ribbon graph (weightings are omitted for clarity).}
    \label{EmbPack}
\end{figure}

Observe that whilst there is a unique packaged ribbon graph arising from a  given vertex and region weighted graph embedded in a pseudo-surface, the converse is not true. This can be seen through how we attach the handles.  
In order to establish a one-to-one correspondence between the class of graphs embedded in pseudo-surfaces and the class of packaged ribbon graphs, we need to use the following equivalence relation from  Definition~14 of \cite{HM}. We say two graphs embedded in a pseudo-surface are related by \emph{stabilisation} if one can be obtained from the other by a finite sequence of removal and addition of handles which does not disconnect any region or coalesce any two regions, and any discs or annuli involved in adding or removing handles are disjoint from the graph. Using Theorem~15 of \cite{HM}, two graphs embedded in pseudo-surfaces with vertex and region weightings correspond to the same packaged ribbon graph if and only if they are related by stabilisation.

\section{A Tutte polynomial for packaged ribbon graphs}\label{shjk}

In this section we introduce a Tutte polynomial for packaged ribbon graphs, $\T(\pG;\boldsymbol{x},\boldsymbol{y})$. We initially define our polynomial as a state sum (Definition~\ref{fhasj}). However, in Theorem~\ref{dcthm} we show that, just as with the classical Tutte polynomial,  it can equivalently be defined through deletion-contraction relations. Moreover,  Theorem~\ref{thun} shows that it is a universal deletion-contraction invariant. In the next section we shall show that $\T(\pG;\boldsymbol{x},\boldsymbol{y})$ subsumes both the surface Tutte polynomial of~\cite{maps1}, and the canonical topological Tutte polynomials of~\cite{HM,KMT}. As a consequence we obtain a deletion-contraction relation for the surface Tutte polynomial.

\begin{definition}\label{fhasj}
    Let $\pG=(\bG,\V,\wV,\F,\wF)$ be a packaged ribbon graph. Define the \emph{packaged surface Tutte polynomial} $\T(\pG;\boldsymbol{x},\boldsymbol{y})$, with variables $\boldsymbol{x}=(x,x_0,x_{1/2},x_1, \ldots)$, $\boldsymbol{y}=(y,y_0,y_{1/2},y_1,\ldots)$, as follows
    \[\T(\pG;\boldsymbol{x},\boldsymbol{y}):= \sum_{A\subseteq E} x^{n(G(\bG^*|A^c;\F))}y^{n(G(\bG|A;\V))}\prod_{\substack{H \text{ cpt. of } \\  G(\bG^*|A^c;\F)}}x_{g(\bG^*,H)}\prod_{\substack{K \text{ cpt. of } \\  G(\bG|A;\V)}}y_{g(\bG,K)}\]
    where cpt. refers to connected component,
    \[g(\bG,K):=\tfrac{1}{2}(2k(K)+e(K)-v(K)+\wV(K)-f(\bG[K])),\]
    \[g(\bG^*,H):=\tfrac{1}{2}(2k(H)+e(H)-v(H)+\wV(H)-f(\bG^*[H])),\]
    here  $\wV(K)=\sum_{v\in V(K)} \wV([v])$, and $\wV(H)$ is defined similarly.
\end{definition}
Note that although the packaged surface Tutte polynomial is defined in an unbounded number of variables, for any given  $\pG$ it is a polynomial in a finite number of variables. We note that  a consequence of Theorem~\ref{asg} is that when $\T(\pG;\boldsymbol{x},\boldsymbol{y})$ is restricted to (non-packaged) ribbon graphs only integrally indexed variables arise. 

\begin{example} \label{ExStateSum}
    Consider the initial packaged ribbon graph $\pG$ depicted in Figure \ref{PolyDC} (i.e. a 2-cycle interlaced with a loop). Every  block in the partitions has size one and every block is mapped to zero. Applying the definition of the packaged surface Tutte polynomial to this packaged ribbon graph yields the following.
    \[\T(\pG;\boldsymbol{x},\boldsymbol{y})= x^3x_1y_0^2 +2x^2x_1y_0+3xyx_0y_0+x^2yx_0y_0^2+y^2x_0y_1 .\]
\end{example}

\bigskip

We defined $\T(\pG;\boldsymbol{x},\boldsymbol{y})$ as a state-sum, however the following theorem shows that, just as with the Tutte polynomial, it can be defined through deletion-contraction relations. This is the key property for the applications of the polynomial in the next sections.
\begin{theorem}\label{dcthm}
    Let $\pG=(\bG,\V,\wV,\F,\wF)$ be a packaged ribbon graph and $e\in E$. Then
	\[\T(\pG;\boldsymbol{x},\boldsymbol{y}) = x^{2-\n}\T(\pG\backslash e;\boldsymbol{x},\boldsymbol{y})+y^{2-\m}\T(\pG/e;\boldsymbol{x},\boldsymbol{y}).\]
    If $\bG$ has no edges, then 
       \[\T(\pG;\boldsymbol{x},\boldsymbol{y})=\prod_{[f]\in \F} x_{(1- \vert [f]\vert +\omega_{\F}([f]))/2} \prod_{[v]\in\V} y_{(1-\vert [v]\vert +\omega_{\V}([v]))/2}\]
    where, for any block $[a]$ of a partition, $\vert [a] \vert$ denotes how many elements are in the block.
\end{theorem}

\begin{proof}
    Before we begin our analysis into how the packaged surface Tutte polynomial changes under deletion and contraction, we will simplify it by expressing it in terms of two polynomials that correspond under duality. Define the following polynomials for any packaged ribbon graph $\pG=(\bG,\V,\omega_\V,\F,\omega_\F)$ with edge subset $A$
    \begin{align*}
	   \Theta(\pG,A;\boldsymbol{x}) & :=x^{n(G(\bG^*|A^c;\F))}     \prod_{\substack{H \text{ cpt. of } \\  G(\bG^*|A^c;\F)}}x_{g(\bG^*,H)}, \\
	   \Phi(\pG,A;\boldsymbol{y}) & := y^{n(G(\bG|A;\V))}\prod_{\substack{K \text{ cpt. of } \\  G(\bG|A;\V)}}y_{g(\bG,K)}.
    \end{align*}

    Consider $\Theta(\pG^*,A^c;\boldsymbol{x})$ where $\pG^*=(\bG^*,\F,\omega_\F,\V,\omega_\V)$. As duality is an involution and $(A^c)^c=E\setminus A^c=A$, we have 
     \begin{align}\label{PhiThetaDual}
         \Theta(\pG^*,A^c;\boldsymbol{x})=x^{n(G(\bG|A;\V))}     \prod_{\substack{H \text{ cpt. of } \\  G(\bG|A;\V)}}x_{g(\bG,H)}=\Phi(\pG,A;\boldsymbol{x}).
     \end{align}
    Using the above we can rewrite $\T(\pG;\boldsymbol{x},\boldsymbol{y})$ in terms of $\Phi$ as follows.
	\begin{equation}\begin{split}\label{TPhiTheta}
	\T(\pG;\boldsymbol{x},\boldsymbol{y}) &=\sum_{A\subseteq E}\Theta(\pG,A;\boldsymbol{x})\Phi(\pG,A;\boldsymbol{y}) \\
	&=\sum_{A\subseteq E\setminus e} \lbrack \Theta(\pG,A;\boldsymbol{x})\Phi(\pG,A;\boldsymbol{y}) + \Theta(\pG,A\cup e;\boldsymbol{x})\Phi(\pG,A\cup e;\boldsymbol{y}) \rbrack \\
        &=\sum_{A\subseteq E\setminus e} \Phi(\pG^*,A^c;\boldsymbol{x})\Phi(\pG,A;\boldsymbol{y}) + \sum_{A\subseteq E\setminus e} \Phi(\pG^*,(A\cup e)^c;\boldsymbol{x})\Phi(\pG,A\cup e;\boldsymbol{y}) .
	\end{split}\end{equation}

     We need to check how the packaged surface Tutte polynomial changes when we delete or contract the four different types of edges (i.e. whether $\n=1$ or 2 and whether $\m=1$ or~2). With this new expression for the polynomial that uses duality, it suffices to look at how $\Phi(\pG,A;\boldsymbol{y})$ changes under deletion and contraction. As $\Phi(\pG,A;\boldsymbol{y})$ is independent from the boundary component partition $\F$, our edge cases are only dependent on the vertex partition $\V$. So we only consider the two edge cases when $\m=1$ and $\m=2$. For the remainder of the proof for the first half of the theorem, we assume that $A\subseteq E\setminus e$.

     We consider how $\Phi(\pG,A;\boldsymbol{y})$ changes under deletion. As $A\subseteq E\setminus e$, then $\pG\ba A^c=(\pG\ba e)\ba A^c$ so $\bG|A=(\bG\ba e)|A$ and $\V, \wV$ are unchanged. So in this case $\Phi(\pG,A;\boldsymbol{y})=\Phi(\pG\ba e,A;\boldsymbol{y})$.

     We want to express $\Phi(\pG,A\cup e;\boldsymbol{y})$ as a polynomial of $\pG/e=(\bG/e,\V',\omega_\V',\F',\omega_\F')$. Here we consider the two edge cases of $\m=2$ and $\m=1$ separately. Suppose $\m=2$, or equivalently $e$ is incident to two vertices $u,v$ where $[u]\neq [v]$. The packaging $G(\bG/e|A;\V')$ has the same number of connected components, one less edge and one less vertex than the packaging $G(\bG|(A\cup e);\V)$. So $n(G(\bG/e|A;\V'))=n(G(\bG|(A\cup e);\V))$. As the number of connected components is unchanged, there is a one-to-one correspondence between the connected components of $G(\bG/e|A;\V')$ and those of $G(\bG|(A\cup e);\V)$. Let $K$ be a connected component of $G(\bG|(A\cup e);\V)$ and $K'$ denote its corresponding connected component of $G(\bG/e|A;\V')$. First assume $e\in E(K)$. Observe that contracting $e$ in $\pG$ before forming the packaging does not change the total weight $\wV(K)$, it decreases the number of edges and the number of vertices in $K$ by one, and does not change the number of boundary components. Therefore $g(\bG,K)=g(\bG/e,K')$. Next if $e\notin E(K)$, then contracting $e$ does not alter $K$, and $g(\bG,K)=g(\bG/e,K')$ is unaffected by contracting $e$. So $\Phi(\pG,A\cup e;\boldsymbol{y})=\Phi(\pG/e,A;\boldsymbol{y})$ when $\m=2$. 

     Now suppose $\m=1$, so either $e$ is a non-loop edge whose incident vertices are in the same block of $\V$ or $e$ is a loop. We first assume the former is true, then the packaging $G(\bG/e|A;\V')$ has the same number of connected components, one less edge and the same number of vertices as the packaging $G(\bG|(A\cup e);\V)$. So $n(G(\bG/e|A;\V'))=n(G(\bG|(A\cup e);\V))-1$. As before, there is a one-to-one correspondence between the connected components of $G(\bG/e|A;\V')$ and those of $G(\bG|(A\cup e);\V)$. Let $K$ be a connected component of $G(\bG|(A\cup e);\V)$ and $K'$ denote its corresponding connected component of $G(\bG/e|A;\V')$. First assume $e\in E(K)$. Observe that contracting $e$ in $\pG$ before forming the packaging  increases the total weight $\wV(K)$ by one, it decreases the number of edges by one, and does not change the number of boundary components or vertices. Therefore, $g(\bG,K)=g(\bG/e,K')$. Next if $e\notin E(K)$, then contracting $e$ does not alter $K$, and $g(\bG,K)=g(\bG/e,K')$. So $\Phi(\pG,A\cup e;\boldsymbol{y})=y\Phi(\pG/e,A;\boldsymbol{y})$.
     
     Finally, assume $e$ is a loop. Then the packaging $G(\bG/e|A;\V')$ has the same number of connected components, one less edge and the same number of vertices as the packaging $G(\bG|(A\cup e);\V)$. So $n(G(\bG/e|A;\V'))=n(G(\bG|(A\cup e);\V))-1$. As before, there is a one-to-one correspondence between the connected components of $G(\bG/e|A;\V)$ and those of $G(\bG|(A\cup e);\V)$. Let $K$ be a connected component of $G(\bG|(A\cup e);\V)$. First assume $e\in E(K)$. Observe that contracting $e$ in $\pG$ before forming the packaging increases the total weight $\wV(K)$ by one, decreases the number of edges by one, and does not change the number of boundary components. Therefore, $g(\bG,K)=g(\bG/e,K')$. Next if $e\notin E(K)$, then $g(\bG,K)$ is unaffected by contracting $e$. So $\Phi(\pG,A\cup e;\boldsymbol{y})=y\Phi(\pG/e,A;\boldsymbol{y})$ when $\m=1$.

     All that remains to prove the first half of the theorem is expressing $\Phi(\pG^*,A^c;\boldsymbol{x})$ and $\Phi(\pG^*,(A\cup e)^c;\boldsymbol{x})$ in terms of $\pG\ba e$ and $\pG/e$. Recall that $A\subseteq E\setminus e$, so $e\in A^c$ and $e\notin (A\cup e)^c$. Therefore by the previous case analysis 
     \begin{equation*}
        \Phi(\pG^*,A^c;\boldsymbol{x})=
        \begin{cases}
             \Phi(\pG^*/e,A^c\setminus e;\boldsymbol{x}) & \text{ if $\mgd=2$},\\
             x\Phi(\pG^*/e,A^c\setminus e;\boldsymbol{x}) & \text{ if $\mgd=1$},
         \end{cases}
     \end{equation*}
    and $\Phi(\pG^*,(A\cup e)^c;\boldsymbol{x})=\Phi(\pG^*\ba e,(A\cup e)^c;\boldsymbol{x})$. By Proposition \ref{dualitydelcon}, $\pG^*/e=\pG\ba e$ and $\pG^*\ba e=\pG/ e$. Also observe that $\mgd=2$ for an edge $e$ in the dual $\pG^*$ is equivalent to the corresponding edge $e$ (here we do abuse notation by not relabelling the dual edges) in $\pG$ having $\np=2$. The same argument also holds for $\mgd=1$ corresponding to $\np=1$ under duality. Hence, 
    \begin{equation*}
        \Phi(\pG^*,A^c;\boldsymbol{x})=
        \begin{cases}
             \Phi((\pG\ba e)^*,A^c\setminus e;\boldsymbol{x}) & \text{ if $\np=2$},\\
             x\Phi((\pG\ba e)^*,A^c\setminus e;\boldsymbol{x}) & \text{ if $\np=1$},
         \end{cases}
     \end{equation*}
    and $\Phi(\pG^*,(A\cup e)^c;\boldsymbol{x})=\Phi((\pG/ e)^*,(A\cup e)^c;\boldsymbol{x})$. Substituting all the relevant deletion and contraction relations for $\Phi$ into (\ref{TPhiTheta}) gives the first half of the theorem.

    Finally for the edgeless case, suppose $\pG$ is an edgeless packaged ribbon graph then the only edge subset to consider is $A=\emptyset$. In which case, $n(G(\bG^*|A^c;\F))=n(G(\bG|A;\V))=0$ and each connected component of $G(\bG|A;\V)$ (resp. $G(\bG^*|A^c;\F)$) is just a block of $\V$ (resp. $\F$). As there are no edges, the number of elements in each block (of either partition) corresponds to how many vertices, and hence how many boundary components, are in the ribbon graph induced by that block. Substituting this into $\T$ gives the second half of the theorem. 
\end{proof}

\begin{figure}[!t]
    \centering
    \input{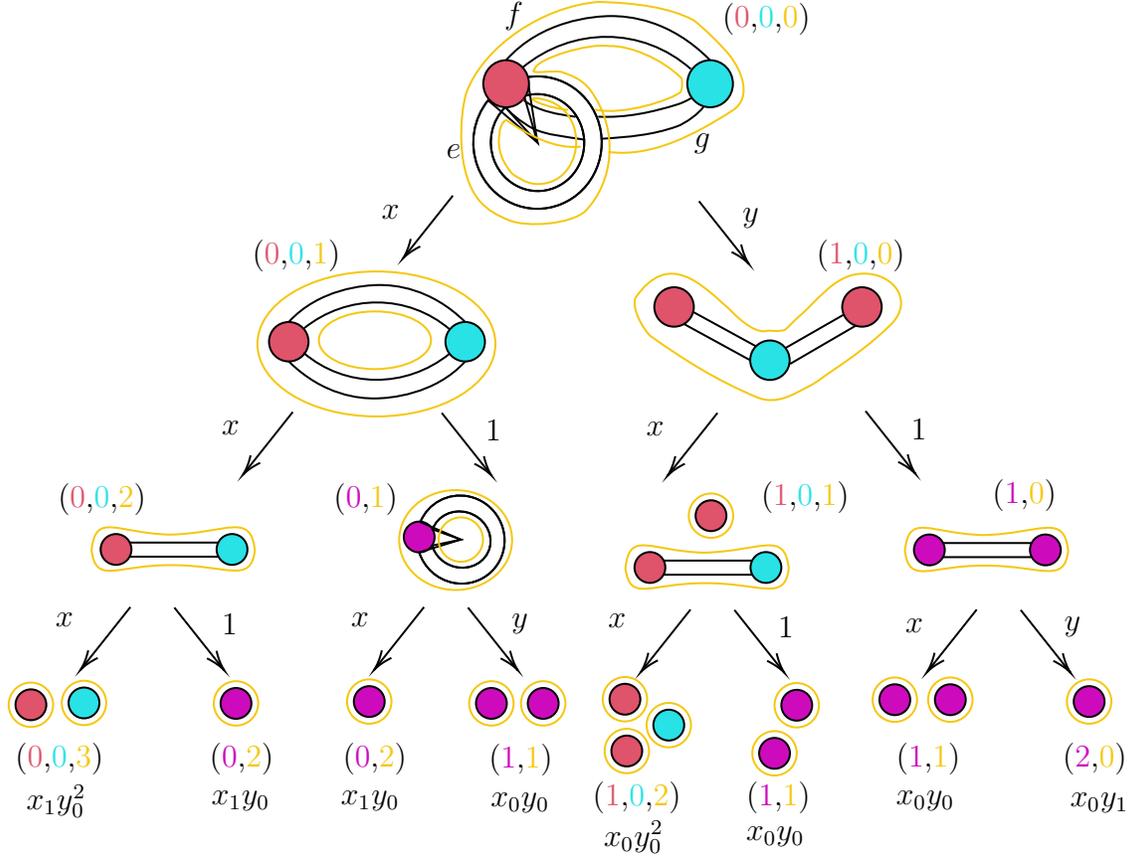}
    \caption{A computation of the packaged surface Tutte polynomial using deletion and contraction.}
    \label{PolyDC}
\end{figure}

\begin{example}\label{sag}
We give the alternative deletion-contraction method of computing the packaged surface Tutte polynomial of Example \ref{ExStateSum} using Theorem \ref{dcthm} in Figure \ref{PolyDC}. Within the diagram, the labels on the arrows are the coefficients for the packaged surface Tutte polynomial applied to the packaged ribbon graph at the head of the arrow. These coefficients arise from the deletion and contraction relation in Theorem \ref{dcthm}. The terms beneath each edgeless packaged ribbon graph are their evaluations under the packaged surface Tutte polynomial. The full evaluation of the polynomial applied to the initial packaged ribbon graph is equal to the sum of the polynomial applied to each edgeless packaged ribbon graph where each term is multiplied by its associated coefficient. Thus 
\[\T(\pG;\boldsymbol{x},\boldsymbol{y})= x^3x_1y_0^2+x^2x_1y_0+x^2x_1y_0+xyx_0y_0+x^2yx_0y_0^2+xyx_0y_0+xyx_0y_0+y^2x_0y_1.\]
\end{example}

In Equation~\eqref{PhiThetaDual} we observed that $\Phi(\pG,A;\boldsymbol{y}) = \Theta(\pG^*,A^c;\boldsymbol{y})$. So by substitution we get the following duality relation that extends the classical duality relation for the Tutte polynomial that $T(G;x,y)=T(G^*;y,x)$ for a plane graph $G$.
\begin{theorem}\label{thup}
    Let $\pG=(\bG,\V,\wV,\F,\wF)$ be a packaged ribbon graph. Then
    \[\T(\pG;\boldsymbol{x},\boldsymbol{y})=\T(\pG^*; \boldsymbol{y},\boldsymbol{x}).\]
\end{theorem}

The following is a universality theorem for the packaged surface Tutte polynomial that is analogous to the universality theorem for the Tutte polynomial as stated in \cite{zbMATH01179517}.

\begin{theorem}\label{thun}
    There is a unique map $U$ from packaged ribbon graphs to $\mathbb{Z}[\alpha,\beta,\kappa,\tau,\boldsymbol{a},\boldsymbol{b}]$, where $\boldsymbol{a}=(a_0,a_1,\ldots)$ and $\boldsymbol{b}=(b_0,b_1,\ldots)$, such that
    \begin{align}\label{ft1}
        U(\pG)=
        \begin{cases}
            \alpha U(\pG\ba e) + \beta U(\pG/e) &\text{ if $\n=1$ and $\m=1$,}\\
            \alpha U(\pG\ba e) + \kappa U(\pG/e) &\text{ if $\n=1$ and $\m=2$,}\\
            \tau U(\pG\ba e) + \beta U(\pG/e) &\text{ if $\n=2$ and $\m=1$,}\\
            \tau U(\pG\ba e) + \kappa U(\pG/e) &\text{ if $\n=2$ and $\m=2$,}
        \end{cases}
    \end{align}
    for every edge $e\in E$, and if $\pG$ is edgeless we have 
    \begin{equation}\label{ft2}
    U(\pG)=\prod_{[f]\in\F} \tau a_{\wF([f])-|[f]|} \prod_{[v]\in\V}\kappa b_{\wV([v])-|[v]|}.\end{equation}
    Furthermore, 
     \begin{equation}\label{ft3} U(\pG)=\tau^{|\F|}\kappa^{|\V|}\T(\pG;\boldsymbol{x},\boldsymbol{y}),\end{equation}
    where $x=\alpha$, $y=\beta$, $x_i= a_{2(1-i)}$, $y_i= b_{2(1-i)}$ for $i=0,1,\ldots$.
\end{theorem}

\begin{proof}
    As every edge type is accounted for by the deletion-contraction relations and the edgeless case is defined, uniqueness is immediate. Thus it is enough to show that~\eqref{ft3} satisfies~\eqref{ft1} and~\eqref{ft2}.
    
    The fact that $U(\pG)\in \mathbb{Z}[\alpha,\beta,\kappa,\tau,\boldsymbol{a},\boldsymbol{b}]$ follows from the degrees of $\alpha$ and $\beta$ being the nullity of a graph, namely the two packagings $G(\bG;\V)$ and $G(\bG^*;\F)$, and the products being taken over the connected components. 

    We check that the expression $\tau^{|\F|}\kappa^{|\V|}\T(\pG;\boldsymbol{x},\boldsymbol{y})$, where $x=\alpha$, $y=\beta$, $x_i= a_{2(1-i)}$, $y_i= b_{2(1-i)}$ for $i=0,1,\dots$, satisfies the deletion-contraction relations in~\eqref{ft1}. Suppose there exists some edge $e\in E$. By Theorem \ref{dcthm}, \[\T(\pG;\boldsymbol{x},\boldsymbol{y})=\alpha^{2-\n}\T(\pG\ba e;\boldsymbol{x},\boldsymbol{y})+\beta^{2-\m}\T(\pG/e;\boldsymbol{x},\boldsymbol{y}).\]
    Additionally, observe that $|\F|$ is unchanged by contracting an edge, and by deleting an edge when $\n=1$. However, $|\F|$ decreases by one when deleting an edge where $\n=2$. Similarly, $|\V|$ is unchanged by deleting an edge, and by contracting an edge where $\m=1$. But $|\V|$ decreases by one by contracting an edge where $\m=2$. Hence, $\tau^{|\F|}\kappa^{|\V|}\T(\pG;\boldsymbol{x},\boldsymbol{y})$ satisfies the deletion-contraction relations as specified in the theorem.
    
    Finally, the edgeless case follows by evaluating~\eqref{ft3}.
    \end{proof}

\section{Connections with other topological Tutte polynomials}\label{aghd}
As discussed in Section~\ref{sint}, the significance of our packaged ribbon graph polynomial is that it unifies two families of topological Tutte polynomials: it specialises to both the surface Tutte polynomial of~\cite{maps1} and the canonical topological Tutte polynomials of~\cite{HM,KMT}. Moreover, it provides a recursive deletion-contraction definition for the surface Tutte polynomial. 
We prove these results in this section. 

To provide context for our results, we begin with a brief overview of the main topological Tutte polynomials in the literature. 
As their exact forms are not needed here, we shall generally omit precise definitions of these polynomials. Individual definitions can be found in the source papers or collectively in \cite[Chapter~27]{zbMATH07553843} or \cite[Section~4]{HM}. The relations between the polynomials discussed here is summarised in Figure~\ref{rhj}.

\begin{figure}
\centering
\scalebox{0.7}{\tikzset{every picture/.style={line width=0.75pt}} 

\begin{tikzpicture}[x=0.75pt,y=0.75pt,yscale=-1,xscale=1]

\draw   (228,1) -- (431,1) -- (431,49) -- (228,49) -- cycle ;
\draw   (84,61) -- (201,61) -- (201,109) -- (84,109) -- cycle ;
\draw   (440,61) -- (601,61) -- (601,109) -- (440,109) -- cycle ;
\draw    (431,28) -- (497.19,59.15) ;
\draw [shift={(499,60)}, rotate = 205.2] [color={rgb, 255:red, 0; green, 0; blue, 0 }  ][line width=0.75]    (10.93,-3.29) .. controls (6.95,-1.4) and (3.31,-0.3) .. (0,0) .. controls (3.31,0.3) and (6.95,1.4) .. (10.93,3.29)   ;
\draw    (228,29) -- (152.86,59.25) ;
\draw [shift={(151,60)}, rotate = 338.07] [color={rgb, 255:red, 0; green, 0; blue, 0 }  ][line width=0.75]    (10.93,-3.29) .. controls (6.95,-1.4) and (3.31,-0.3) .. (0,0) .. controls (3.31,0.3) and (6.95,1.4) .. (10.93,3.29)   ;
\draw   (131,162) -- (271,162) -- (271,209) -- (131,209) -- cycle ;
\draw  [color={rgb, 255:red, 0; green, 0; blue, 0 }  ,draw opacity=1 ] (393,151) -- (509,151) -- (509,184) -- (393,184) -- cycle ;
\draw  [color={rgb, 255:red, 0; green, 0; blue, 0 }  ,draw opacity=1 ] (546,151) -- (649,151) -- (649,183) -- (546,183) -- cycle ;
\draw [color={rgb, 255:red, 0; green, 0; blue, 0 }  ,draw opacity=1 ]   (570,108) -- (589.12,147.2) ;
\draw [shift={(590,149)}, rotate = 244] [color={rgb, 255:red, 0; green, 0; blue, 0 }  ,draw opacity=1 ][line width=0.75]    (10.93,-3.29) .. controls (6.95,-1.4) and (3.31,-0.3) .. (0,0) .. controls (3.31,0.3) and (6.95,1.4) .. (10.93,3.29)   ;
\draw [color={rgb, 255:red, 0; green, 0; blue, 0 }  ,draw opacity=1 ]   (455,110) -- (439.76,147.15) ;
\draw [shift={(439,149)}, rotate = 292.31] [color={rgb, 255:red, 0; green, 0; blue, 0 }  ,draw opacity=1 ][line width=0.75]    (10.93,-3.29) .. controls (6.95,-1.4) and (3.31,-0.3) .. (0,0) .. controls (3.31,0.3) and (6.95,1.4) .. (10.93,3.29)   ;
\draw  [dash pattern={on 4.5pt off 4.5pt}]  (150,110) -- (177.07,161.23) ;
\draw [shift={(178,163)}, rotate = 242.15] [color={rgb, 255:red, 0; green, 0; blue, 0 }  ][line width=0.75]    (10.93,-3.29) .. controls (6.95,-1.4) and (3.31,-0.3) .. (0,0) .. controls (3.31,0.3) and (6.95,1.4) .. (10.93,3.29)   ;
\draw    (439,77) -- (231.86,160.25) ;
\draw [shift={(230,161)}, rotate = 338.1] [color={rgb, 255:red, 0; green, 0; blue, 0 }  ][line width=0.75]    (10.93,-3.29) .. controls (6.95,-1.4) and (3.31,-0.3) .. (0,0) .. controls (3.31,0.3) and (6.95,1.4) .. (10.93,3.29)   ;
\draw   (46,265) -- (166,265) -- (166,312) -- (46,312) -- cycle ;
\draw   (261,264) -- (419,264) -- (419,312) -- (261,312) -- cycle ;
\draw    (141,210) -- (116.84,262.19) ;
\draw [shift={(116,264)}, rotate = 294.84] [color={rgb, 255:red, 0; green, 0; blue, 0 }  ][line width=0.75]    (10.93,-3.29) .. controls (6.95,-1.4) and (3.31,-0.3) .. (0,0) .. controls (3.31,0.3) and (6.95,1.4) .. (10.93,3.29)   ;
\draw    (270,209) -- (317.64,260.53) ;
\draw [shift={(319,262)}, rotate = 227.25] [color={rgb, 255:red, 0; green, 0; blue, 0 }  ][line width=0.75]    (10.93,-3.29) .. controls (6.95,-1.4) and (3.31,-0.3) .. (0,0) .. controls (3.31,0.3) and (6.95,1.4) .. (10.93,3.29)   ;
\draw [color={rgb, 255:red, 0; green, 0; blue, 0 }  ,draw opacity=1 ]   (420,185) -- (381.9,260.22) ;
\draw [shift={(381,262)}, rotate = 296.86] [color={rgb, 255:red, 0; green, 0; blue, 0 }  ,draw opacity=1 ][line width=0.75]    (10.93,-3.29) .. controls (6.95,-1.4) and (3.31,-0.3) .. (0,0) .. controls (3.31,0.3) and (6.95,1.4) .. (10.93,3.29)   ;
\draw [color={rgb, 255:red, 0; green, 0; blue, 0 }  ,draw opacity=1 ]   (555,184) -- (421.58,286.78) ;
\draw [shift={(420,288)}, rotate = 322.39] [color={rgb, 255:red, 0; green, 0; blue, 0 }  ,draw opacity=1 ][line width=0.75]    (10.93,-3.29) .. controls (6.95,-1.4) and (3.31,-0.3) .. (0,0) .. controls (3.31,0.3) and (6.95,1.4) .. (10.93,3.29)   ;
\draw   (279,351) -- (420,351) -- (420,399) -- (279,399) -- cycle ;
\draw    (350,312) -- (350,348) ;
\draw [shift={(350,350)}, rotate = 270] [color={rgb, 255:red, 0; green, 0; blue, 0 }  ][line width=0.75]    (10.93,-3.29) .. controls (6.95,-1.4) and (3.31,-0.3) .. (0,0) .. controls (3.31,0.3) and (6.95,1.4) .. (10.93,3.29)   ;
\draw    (131,312) -- (276.11,361.36) ;
\draw [shift={(278,362)}, rotate = 198.79] [color={rgb, 255:red, 0; green, 0; blue, 0 }  ][line width=0.75]    (10.93,-3.29) .. controls (6.95,-1.4) and (3.31,-0.3) .. (0,0) .. controls (3.31,0.3) and (6.95,1.4) .. (10.93,3.29)   ;
\draw [color={rgb, 255:red, 0; green, 0; blue, 0 }  ,draw opacity=1 ]   (611,184) -- (422.43,368.6) ;
\draw [shift={(421,370)}, rotate = 315.61] [color={rgb, 255:red, 0; green, 0; blue, 0 }  ,draw opacity=1 ][line width=0.75]    (10.93,-3.29) .. controls (6.95,-1.4) and (3.31,-0.3) .. (0,0) .. controls (3.31,0.3) and (6.95,1.4) .. (10.93,3.29)   ;
\draw   (300,433) -- (401,433) -- (401,481) -- (300,481) -- cycle ;
\draw    (350,399) -- (350,430) ;
\draw [shift={(350,432)}, rotate = 270] [color={rgb, 255:red, 0; green, 0; blue, 0 }  ][line width=0.75]    (10.93,-3.29) .. controls (6.95,-1.4) and (3.31,-0.3) .. (0,0) .. controls (3.31,0.3) and (6.95,1.4) .. (10.93,3.29)   ;
\draw [color={rgb, 255:red, 0; green, 0; blue, 0 }  ,draw opacity=1 ]   (489,185) -- (412.85,347.19) ;
\draw [shift={(412,349)}, rotate = 295.15] [color={rgb, 255:red, 0; green, 0; blue, 0 }  ,draw opacity=1 ][line width=0.75]    (10.93,-3.29) .. controls (6.95,-1.4) and (3.31,-0.3) .. (0,0) .. controls (3.31,0.3) and (6.95,1.4) .. (10.93,3.29)   ;

\draw (296,26.4) node [anchor=north west][inner sep=0.75pt]    {$\boldsymbol{T}(\mathbb{\underline{G}} ;\boldsymbol{x} ,\boldsymbol{y})$};
\draw (276,5) node [anchor=north west][inner sep=0.75pt]   [align=left] {Packaged Tutte};
\draw (109,86.4) node [anchor=north west][inner sep=0.75pt]    {$\mathcal{T}(\mathbb{G} ;\boldsymbol{x} ,\boldsymbol{y})$};
\draw (97,65) node [anchor=north west][inner sep=0.75pt]   [align=left] {Surface Tutte};
\draw (461,85.4) node [anchor=north west][inner sep=0.75pt]    {$T_{p}{}_{s}(\mathbb{G} ;w,x,y,z)$};
\draw (448,65) node [anchor=north west][inner sep=0.75pt]   [align=left] {Pseudo-surface Tutte};
\draw (150,187.4) node [anchor=north west][inner sep=0.75pt]    {$K(\mathbb{G} ;x,y,a,b)$};
\draw (171,167) node [anchor=north west][inner sep=0.75pt]   [align=left] {Krushkal};
\draw (397,159.4) node [anchor=north west][inner sep=0.75pt]  [color={rgb, 255:red, 0; green, 0; blue, 0 }  ,opacity=1 ]  {$T_{cp}{}_{s}(\mathbb{G} ;w,x,y)$};
\draw (553,158.4) node [anchor=north west][inner sep=0.75pt]  [color={rgb, 255:red, 0; green, 0; blue, 0 }  ,opacity=1 ]  {$T_{s}(\mathbb{G} ;x,y,z)$};
\draw (63,291.4) node [anchor=north west][inner sep=0.75pt]    {$L(\mathbb{G} ;x,y,z)$};
\draw (63,271) node [anchor=north west][inner sep=0.75pt]   [align=left] {Las Vergnas};
\draw (298,287.4) node [anchor=north west][inner sep=0.75pt]    {$R(\mathbb{G} ;x,y,z)$};
\draw (276,269) node [anchor=north west][inner sep=0.75pt]   [align=left] {Bollob\'as--Riordan};
\draw (315,376.4) node [anchor=north west][inner sep=0.75pt]    {$R(\mathbb{G} ;x,y)$};
\draw (303,356) node [anchor=north west][inner sep=0.75pt]   [align=left] {Ribbon Graph};
\draw (333,437) node [anchor=north west][inner sep=0.75pt]   [align=left] {Tutte};
\draw (316,459.4) node [anchor=north west][inner sep=0.75pt]    {$T( G;x,y)$};
\draw (68,118) node [anchor=north west][inner sep=0.75pt]   [align=left] {cellularly\\embedded};

\end{tikzpicture}}
\caption{The relationship between  $\T(\pG;\boldsymbol{x},\boldsymbol{y})$ and other topological Tutte polynomials.}
\label{rhj}
\end{figure}
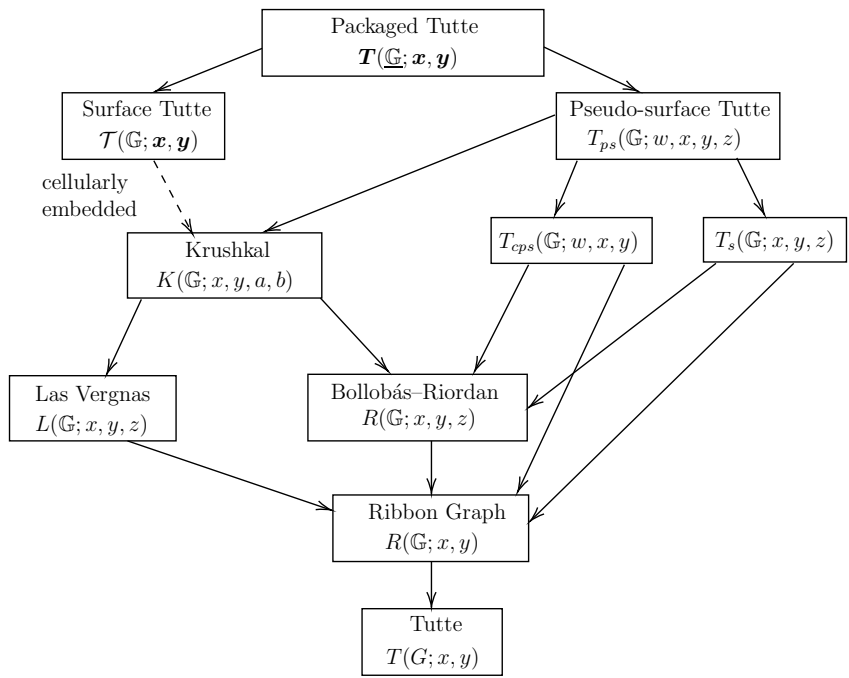

 Our first polynomial  is the \emph{Las~Vergnas polynomial} $L(\bG;x,y,z)$ introduced in the 1980 paper~\cite{zbMATH03722664}. This is a polynomial of graphs cellularly embedded in surfaces that arises as a special instance of his Tutte polynomial of matroid perspectives. 
Unfortunately, the  Las~Vergnas polynomial did not receive the attention it deserved and it took another 20 years until the introduction of the \emph{Bollob\'as--Riordan polynomial} of a ribbon graph in~\cite{bollobasriordanpoly,zbMATH01801590},  $R(\bG;x,y,z)$, for the field of topological Tutte polynomials to begin in earnest. Much of the development of this polynomial has focussed on the \emph{ribbon graph polynomial},  $R(\bG;x,y)$, which arises as a 2-variable specialisation of the  Bollob\'as--Riordan polynomial. 
The \emph{Krushkal polynomial}, $K(\bG;x,y,a,b)$,  was introduced around ten years later in~\cite{krushkalpoly}. This is a polynomial of graphs that are (not-necessarily cellularly) embedded in surfaces. It was shown to specialise to the  Bollob\'as--Riordan polynomial in~\cite{krushkalpoly} and the Las~Vergnas polynomial in~\cite{zbMATH06127547}. All these polynomials specialise to the Tutte polynomial and in more than one way.

Of the above polynomials, only the ribbon graph polynomial has a full deletion-contraction definition akin to that for the Tutte polynomial (see the discussion around~\eqref{tp2}). However, by enlarging their domains it is possible to find extensions of them that do have full deletion-contraction definitions. This was first done for the Las~Vergnas polynomial in~\cite{zbMATH06473566} by extending its domain to graphs in pseudo-surfaces; then, using the theory of \emph{canonical Tutte polynomials}, for the Bollob\'as--Riordan polynomial and the Krushkal polynomial of cellularly embedded graphs in~\cite{zbMATH06473566} by extending their domains to ribbon graphs whose vertex set is partitioned; and for the full Krushkal polynomial in~\cite{HM} by extending its domain to graphs in pseudo-surfaces (see also the delta-matroid extension in~\cite{zbMATH06951578}).  This Krushkal polynomial extension is known as the \emph{Tutte polynomial of a coloured ribbon graph} or the \emph{Tutte polynomial of graphs in pseudo-surfaces} depending upon context.
In our notation, it is defined as follows.

\begin{definition}\label{pndfg}
Let $\bG=(V,E)$ be a  ribbon graph with a partition $\V$ of its vertex set and $\F$ a partition of its set of boundary components. Then 
\[  T_{ps}(\bG;w,x,y,z):=  \sum_{A\subseteq E}   w^{r_1(E) -r_1(A) } x^{r_2(E)-r_2(A)} y^{r_3(A)}  z^{r_4(A)} ,  \]
where
\begin{align*}
r_1(A) &:= r(G(\bG|A;\V)),     \\
r_2(A) &:= r(\bG|A)+g(\bG|A)- r(G(\bG|A;\V)),\\
r_3(A) &:=      r(G(\bG^*;\F)) - r(G(\bG^*|A^c;\F)),\\
 r_4(A) &:= |A| +r(G(\bG^*|A^c;\F)) - r(G(\bG^*;\F)) -r(\bG|A)-g(\bG|A).
 \end{align*}
\end{definition}
All of the polynomials mentioned above can be recovered from $T_{ps}(\bG;w,x,y,z)$. See~\cite{HM} for details. 
We note two significant specialisations of this polynomial:
$T_{cps}(\bG;w,x,y)$ and $T_{s}(\bG;x,y,z)$. These are polynomials for graphs cellularly embedded in pseudo-surfaces, and emebdded in surfaces, respectively. They are related by duality and are extensions of the Bollob\'as--Riordan polynomial that have full deletion-contraction definitions.

Finally, the most recently defined topological Tutte polynomial we consider is the \emph{surface Tutte polynomial} introduced in~\cite{maps1} and extended to non-orientable surfaces in~\cite{maps2}. This polynomial was motivated by Tutte's introduction of the dichromate as a polynomial that counts both the number of  proper $k$-colourings and the number of nowhere-zero $\mathbb{Z}_k$-flows in a graph. In particular, the surface Tutte polynomial  counts both the  local flows and local tensions in a graph cellularly embedded in a surface (this aspect of the polynomial is discussed in Section~\ref{daha}).  
In our notation, the surface Tutte polynomial is defined as follows. 

\begin{definition}\label{bklih}
Let $\bG$ be an orientable ribbon graph. Then the \emph{surface Tutte polynomial} $\Tt(\bG;\boldsymbol{x},\boldsymbol{y})$, where $\boldsymbol{x}=(x,x_0,x_1,\ldots),\boldsymbol{y}=(y,y_0,y_1,\ldots)$, is defined as 
    \[\Tt(\bG;\boldsymbol{x},\boldsymbol{y}):=\sum_{A\subseteq E} x^{n(\bG^*|A^c)}y^{n(\bG|A)}
    \prod_{\substack{\text{cpt. } \\ \bG_i \text{ of } \bG/A}} x_{g(\bG_i)} 	
    \prod_{\substack{\text{cpt. } \\ \bG_j \text{ of } \bG|A}} y_{g(\bG_j)}\]
where $A^c=E\setminus A$.
\end{definition}
The surface Tutte polynomial specialises to the Krushkal polynomial of a cellularly embedded graph and therefore to the Bollob\'{a}s--Riordan, Las~Vergnas and Tutte polynomial (see~\cite{maps1} for details). However, it does not have a known full deletion-contraction definition, and it does not specialise to $T_{ps}(\bG;w,x,y,z)$ or any of the  extensions of Krushkal, Bollob\'{a}s--Riordan, or Las~Vernas polynomials that have deletion-contraction definitions.  

The following theorem shows that \emph{all} of the above topological Tutte polynomials  can be recovered from our polynomial $\Tt(\bG;\boldsymbol{x}, \boldsymbol{y})$.

\begin{theorem}\label{asg}
Both the surface Tutte polynomial and the Tutte polynomial of graphs in pseudo-surfaces can be recovered from the Tutte polynomial of a packaged ribbon graph, as follows.
\begin{enumerate}
\item Let $\bG=(V,E)$ be an orientable ribbon graph and $F$ be its set of boundary components.  Also, let $\pG $ be the  packaged ribbon $(\bG, \{  \{v\}:v\in V \} , 0_V, \{ \{f\}:f\in F \} , 0_F )$, where  $0_{\cdot}$ denotes the zero-map. Then,
\[ \T(\pG;\boldsymbol{x},\boldsymbol{y})=\Tt(\bG;\boldsymbol{x},\boldsymbol{y}) .\]

\item 
Let $\bG=(V,E)$ be an orientable  ribbon graph with a partition $\V$ of its vertex set and $\F$ a partition of its set of boundary components. Furthermore, let $\pG$ be the packaged ribbon graph obtained from this by weighting each block of the partitions with zero. Then when 
$x=1$, $y=1$, $x_g=cd^{-g}$, $y_g=ab^{-g}$ we have, 
\[
T_{ps}(\bG;a,b,c,d) =
(ab)^{k(G(\bG;\V))} 
    b^{ (|E|-|\V|-f(\bG))/2}
    c^{-k(G(\bG^*;\F))  }
    d^{(|E|-|\F|+v(\bG))/2}
    \T(\pG;\boldsymbol{x},\boldsymbol{y}) .
\]
\end{enumerate}

\end{theorem}
\begin{proof}
The first item is immediate. For the second,  when $x=1$, $y=1$, $x_g=cd^{-g}$, $y_g=ab^{-g}$, 
   \[\T(\pG;\boldsymbol{x},\boldsymbol{y})= \sum_{A\subseteq E} 
   \prod_{\substack{H \text{ cpt. of } \\  G(\bG^*|A^c;\F)}}   c  d^{-g(\bG^*,H)}
   \prod_{\substack{K \text{ cpt. of } \\  G(\bG|A;\V)}}   a b^{-g(\bG,K)}.\]
By expanding the exponents, we can write this as
\begin{multline*}
 \sum_{A\subseteq E} 
   \prod_{\substack{H \text{ cpt. of } \\  G(\bG^*|A^c;\F)}}   c  d^{ (-2k(H)-e(H)+v(H)+f(\bG^*[H])) /2}
   \prod_{\substack{K \text{ cpt. of } \\  G(\bG|A;\V)}}  a b^{ (-2k(K)-e(K)+v(K)+f(\bG[K])) /2 }
 \\
= \sum_{A\subseteq E} 
   c^{k(G(\bG^*|A^c;\F))}  d^{ (-2k(G(\bG^*|A^c;\F))-e(G(\bG^*|A^c;\F))+v(G(\bG^*|A^c;\F))+f(\bG^*|A^c)) /2}
   \\
 a^{k(G(\bG|A;\V))}  b^{ (-2k(G(\bG|A;\V))-e(G(\bG|A;\V))+v(G(\bG|A;\V))+f(\bG|A)) /2 }
\end{multline*}

Making use of $r(\bG|A)+g(\bG|A)=(e(\bG|A)+ v(\bG|A)-f(\bG|A))/2$   (by Euler's formula), and that $ f(\bG^*|A^c) = f(\bG|A)$ (as the ribbon graphs have the same boundary) this equals
    \begin{multline*}  a^{-k(G(\bG;\V))} 
    b^{(v(G(\bG;\V))-e(\bG)+f(\bG)-2k(G(\bG;\V)))/2}
    c^{k(G(\bG^*;\F))  }
    d^{(v(G(\bG^*;\F))-e(\bG)-v(\bG))/2}
\\
 \sum_{A\subseteq E}  a^{r_1(E) -r_1(A) } b^{r_2(E)-r_2(A)} c^{r_3(A)}  d^{r_4(A)} ,
\end{multline*}
and the result follows.
\end{proof}

As $T_{ps}(\bG;w,x,y,z)$ can be recovered from  $\T(\pG;\boldsymbol{x},\boldsymbol{y})$, the evaluations in \cite[Section~4]{HM} immediately give the following.
\begin{corollary}\label{dbahj}
Each of the Krushkal polynomial, Las Vergnas polynomial, and Bollob\'{a}s--Riordan polynomial, and all of their recursively defined extensions from~\cite{HM} can be recovered from  $\T(\pG;\boldsymbol{x},\boldsymbol{y})$.
\end{corollary}

The next result is an immediate corollary of Threorems~\ref{dcthm} and~\ref{asg}. We state it as a theorem for emphasis.
\begin{theorem}\label{thre}
The surface Tutte polynomial can be computed recursively through deletion-contraction relations.
\end{theorem}

\section{Recursively counting local flows and local tensions}\label{daha}

The surface Tutte polynomial has been shown to count the number of local flows and local tensions for a ribbon graph~\cite{maps1}. We now apply our results  to obtain deletion-contraction relations for the numbers of  local flows and local tensions.

Let  $\bG=(V,E)$ be an oriented ribbon graph and $F$ be its set of boundary components. Each boundary component inherits an orientation from $\bG$.  Furthermore,  direct each edge of $\bG$. Each edge thus consists of two half-edges, one of which is the head and the other the tail. If $h$ is a half-edge we set
\begin{equation*}
        a(h):=
        \begin{cases}
          1 &\text{if $h$ is the head of $e$},\\
          -1 &\text{if $h$ is the tail of $e$}.
        \end{cases}
    \end{equation*}
Each edge contains two segments that are part of  the boundary components of $\bG$. We call these the \emph{edge-sides}. If an edge is directed, then its edge-sides are directed in the same direction. For each edge, the direction of one edge-side will agree with the orientation of the boundary component it sits on, the other will disagree. For an edge-side $h$ we set 
\begin{equation*}
        b(h):=
        \begin{cases}
          1 &\text{if the direction of $h$ disagrees with the boundary orientation},\\
          -1 &\text{if the direction of $h$ agrees with the boundary orientation}.
        \end{cases}
    \end{equation*}
Now suppose that 
$\Gamma$ is a finite group and $\gamma:E\to\Gamma$ is a mapping that assigns a group element to each edge. 
 We say $\gamma$ is a \emph{local $\Gamma$-flow} if for each vertex $v$ we have that 
 \[(\gamma(e_1))^{a(h_1)}(\gamma(e_2))^{a(h_2)}\dotsc (\gamma(e_n))^{a(h_n)} = 1,\]
where $h_1, h_2, ..., h_n$ are the cyclically ordered incident half-edges incident to $v$, and each half-edge $h_i$ belongs to the edge $e_i$.
It is a \emph{nowhere-identity local $\Gamma$-flow} for $\bG$ if it is a local $\Gamma$-flow  and no edge is mapped to the identity element. 

Similarly, $\gamma$ is a \emph{local $\Gamma$-tension} if for each boundary component  $f$ we have that 
 \[(\gamma(e_1))^{b(h_1)}(\gamma(e_2))^{b(h_2)}\dotsc (\gamma(e_n))^{b(h_n)} = 1,\]
where $h_1, h_2, ..., h_n$ are the cyclically ordered edge-sides met when following $f$, and each edge-side $h_i$ belongs to the edge $e_i$.
It is a \emph{nowhere-identity local $\Gamma$-tension} for $\bG$ if it is a local   $\Gamma$-tension and no edge is mapped to the identity element.

We denote the number of local $\Gamma$-flows by $q_\Gamma^1(\bG)$ and the number of nowhere-identity local $\Gamma$-flows by $q_\Gamma(\bG)$. 
Similarly, we denote the number of local $\Gamma$-tensions by $p_\Gamma^1(\bG)$ and the number of nowhere-identity local $\Gamma$-tensions by $p_\Gamma(\bG)$. 
 It was shown in~\cite{maps1} that all four of these counts are independent of the choice of  edge directions for $\bG$.
The following formulas  for $q_\Gamma^1(\bG)$ and $q_\Gamma(\bG)$ were also given in this reference. Suppose $\bG$ is a connected  ribbon graph, and $\Gamma$ is a finite group whose irreducible representations have dimensions $n_1,...,n_k$. Then
\[q^1_\Gamma(\bG)=|\Gamma|^{e(\bG)-v(\bG)}\sum_{k} n_{k}^{f(\bG)-e(\bG)+v(\bG)},\]
    and 
\[q_\Gamma(\bG)=\sum_{A\subseteq E}(-1)^{|A^c|}q_\Gamma^1(\bG|A).\]

We state the corresponding results for $p_\Gamma^1(\bG)$ and $p_\Gamma(\bG)$  using the observation from~\cite{maps1} that $p_\Gamma^1(\bG)=q_\Gamma^1(\bG^*)$ and $p_\Gamma(\bG)=q_\Gamma(\bG^*)$. Suppose $\bG$ is a connected ribbon graph, and $\Gamma$ is a finite group whose irreducible representations have dimensions $n_1,...,n_k$. Then
\[p^1_\Gamma(\bG)=|\Gamma|^{e(\bG)-f(\bG)}\sum_{k} n_{k}^{v(\bG)-e(\bG)+f(\bG)},\]
    and 
\[p_\Gamma(\bG)=\sum_{A\subseteq E}(-1)^{|A^c|}p_\Gamma^1(\bG/A^c).\]

\medskip 

We will now show that the packaged surface Tutte polynomial counts local flows and tensions, and use this to obtain  deletion-contraction relations for the counts. For this we use the following notation.

Given a packaged ribbon graph $\pG=(\bG,\V,\wV,\F,\wF)$, we denote its associated weighted vertex packaged ribbon graph by $\pG_\V:=(\bG,\V,\wV)$. Similarly, we denote its associated weighted boundary component packaged ribbon graph by $\pG_\F:=(\bG,\F,\wF)$. 
Where it makes sense, the terminology and notation for a packaged ribbon graph  $\pG$ is extended to $\pG_\V$ and $\pG_\F$ in the obvious way. 
Now suppose $K$ is a subgraph of the packaging $G(\bG;\V)$, we define \[\pG_\V[K]:=(\bG[K],\V',\wV'),\] where $\V'$ is the restriction of $\V$ to the vertex set of $\bG[K]$, and $\wV'$ is the restriction of $\wV$ to the blocks of $\V'$. Now suppose $H$ is a subgraph of the packaging $G(\bG^*;\F)$. We define \[\pG_\F[H]:=((\bG^*[H])^*,\F',\wF').\] 
Since, by Proposition~\ref{dualitydelcon},  $(\bG^*[H])^*$ can be obtained from $\bG$ by contracting edges and removing some number (potentially zero) of isolated vertices, each boundary component of $(\bG^*[H])^*$ naturally corresponds to one in $\bG$. So we let $\F'$  denote the restriction of $\F$ to the set of all boundary components in $(\bG^*[H])^*$ under this correspondence, and $\wF'$ denote the restriction of $\wF$ to the blocks of $\F'$.

\begin{definition}
    Let $\pG=(\bG,\V,\wV,\F,\wF)$ be a packaged ribbon graph, and $\Gamma$ be a finite group whose irreducible representations have dimensions $n_1,...,n_k$. If $G(\bG;\V)$ is connected, then we define
    \[q_\Gamma^1(\pG_\V):=|\Gamma|^{e(\bG)-|\V|}\sum_{k} n_{k}^{f(\bG)-e(\bG)+|\V|-\wV(\bG)}.\]
    If $G(\bG;\V)$ is disconnected, then we define 
    \[q_\Gamma^1(\pG_\V):=\prod_{\substack{K \text{ cpt. of } \\  G(\bG;\V)}}q_\Gamma^1(\pG_\V[K]).\]
    Finally, we define
    \[q_\Gamma(\pG_\V):= \sum_{A\subseteq E} (-1)^{|A^c|}q_\Gamma^1( (\pG|A)_\V ).\]
\end{definition}

\begin{definition}
    Let $\pG=(\bG,\V,\wV,\F,\wF)$ be a packaged ribbon graph, and $\Gamma$ be a finite group whose irreducible representations have dimensions $n_1,...,n_k$. If $G(\bG^*;\F)$ is connected, then
    \[p_\Gamma^1(\pG_\F):=|\Gamma|^{e(\bG)-|\F|}\sum_{k} n_{k}^{v(\bG)-e(\bG)+|\F|-\wF(\bG)},\]
    If $G(\bG^*;\F)$ is disconnected, then
    \[p_\Gamma^1(\pG_\F):=\prod_{\substack{H \text{ cpt. of } \\  G(\bG^*;\F)   }}p_\Gamma^1(\pG_\F[H]).\]
   Finally, we define
    \[p_\Gamma(\pG_\F):= \sum_{A\subseteq E} (-1)^{|A^c|}p_\Gamma^1( (\pG/A^c)_\F).\] 
\end{definition}

\begin{lemma}\label{jkk}
 Let  $\pG=(\bG,\V,\wV,\F,\wF)$ be a packaged ribbon graph and suppose its  dual is  $\pG^*=(\bG^*,\V^*,\wVdual,\F^*,\wFdual)$, then
 $q_\Gamma(\pG_\V)   = p_\Gamma(\pG^*_{\F^*})  $ and $ p_\Gamma(\pG_{\F}) = q_\Gamma(\pG^*_{\V^*})  $.
\end{lemma}
\begin{proof}
First note that  $G((\bG^*)^*;\F^*)$ is isomorphic to $G(\bG;\V)$. Then 
it is easily seen that $q^1_\Gamma(\pG_\V)   = p^1_\Gamma(\pG^*_{\F^*})  $ when $G((\bG^*)^*;\F^*)$ is connected, and the disconnected case follows from this:
 \[p^1_\Gamma(\pG^*_{\F^*}) 
 = \prod_{\substack{H \text{ cpt. of } \\  G((\bG^*)^*;\F^*) }} p_\Gamma^1((\pG^*_{\F^*}[H])
 = \prod_{\substack{H \text{ cpt. of } \\  G(\bG;\V) }} q_\Gamma^1((\pG_{\V}[H])) = q^1_\Gamma(\pG_\V) .
\]
By an application of Proposition~\ref{dualitydelcon} it follows that $q_\Gamma(\pG_\V)   = p_\Gamma(\pG^*_{\F^*})  $. 
Finally,  the identity  $ p_\Gamma(\pG_{\F}) = q_\Gamma(\pG^*_{\V^*})  $ follows since duality is an involution.
\end{proof}

\begin{proposition}
Let $\bG=(V,E)$ be a connected orientable ribbon graph and $F$ be its set of boundary components.  Also, let $\pG $ be the  packaged ribbon graph $(\bG, \{  \{v\}:v\in V \} , 0_V, \{ \{f\}:f\in F \} , 0_F )$, where  $0_{\cdot}$ denotes the zero-map. Then, $q_\Gamma^1(\pG_\V)=q_\Gamma^1(\bG)$, $q_\Gamma(\pG_\V)=q_\Gamma(\bG)$, $p_\Gamma^1(\pG_\F)=p_\Gamma^1(\bG)$, and $p_\Gamma(\pG_\F)=p_\Gamma(\bG)$.
In particular, $q_\Gamma(\pG_\V)$ gives  the number of nowhere-identity local $\Gamma$-flows in $\bG$, and  $p_\Gamma(\pG_\F)$ gives  the number of nowhere-identity local $\Gamma$-tensions in $\bG$.  
\end{proposition}
\begin{proof}
That $q_\Gamma^1(\pG_\V)=q_\Gamma^1(\bG)$ and $p_\Gamma^1(\pG_\F)=p_\Gamma^1(\bG)$ is immediate. 
Since contraction does not change $\F$ and deletion does not change $\V$ it follows that  
$p_\Gamma(\pG_\V)=p_\Gamma(\bG)$ and $q_\Gamma(\pG_\F)=q_\Gamma(\bG)$.
\end{proof}

\begin{theorem}\label{flowspecial}
  Let $\pG=(\bG,\V,\wV,\F,\wF)$ be a packaged ribbon graph and $\Gamma$ be a finite group whose irreducible representations have dimensions $n_1,...,n_k$. Then the following identities hold.
  \begin{enumerate}
      \item When $x=1$, $y=-|\Gamma|$, $x_{g} =1$, and $y_{g} = -\frac{1}{|\Gamma|}\sum_{i=1}^{k}n_i^{2-2g}$, we have
    \[q_\Gamma(\pG_\V)= (-1)^{e(\bG)-|\V|}\T(\pG;\boldsymbol{x},\boldsymbol{y}).\]
      \item When $x=-|\Gamma|$, $y=1$, $x_{g} =-\frac{1} {|\Gamma|}\sum_{i=1}^{k}n_i^{2-2g}$, and $y_{g} = 1$, we have 
     \[p_\Gamma(\pG_\F)=(-1)^{e(\bG)-|\F|}\T(\pG;\boldsymbol{x},\boldsymbol{y}).\]
  \end{enumerate}
\end{theorem}

\begin{proof}
    \begin{align*}
    q_\Gamma (\pG_\V) 
    &= \sum_{A\subseteq E} (-1)^{|A^c|} \prod_{\substack{K \text{ cpt. of } \\  G(\bG|A;\V)}} |\Gamma|^{e(K)-v(K)} \sum_k n_k^{f(\bG[K])-e(K)+v(K)-\wV(K)}
 \\
    &= \sum_{A\subseteq E} (-1)^{|A^c|} |\Gamma|^{e(\bG|A)-v(G(\bG|A;\V))}\prod_{\substack{K \text{ cpt. of } \\  G(\bG;\V)}} \sum_k n_k^{f(\bG[K])-e(K)+v(K)-\wV(K)}
    \\  & 
       =  (-1)^{e(\bG)-|\V|} \sum_{A\subseteq E} ( -|\Gamma|)^{n(G(\bG|A;\V))}  \prod_{\substack{K \text{ cpt. of } \\  G(\bG|A;\V)}}
    -\frac{1}{|\Gamma|}\sum_{i=1}^{k}n_i^{2-2g (\bG,K)},      
    \end{align*}
    which is $(-1)^{e(\bG)-|\V|}\T(\pG;\boldsymbol{x},\boldsymbol{y})$ when  $x=1, y=-|\Gamma|, x_g =1, y_{g} = -\frac{1}{|\Gamma|}\sum_{i=1}^{k}n_i^{2-2g}$.
    
    The result for $p_\Gamma(\pG_\F)$ follows from the first item by Lemma~\ref{jkk} and Theorem~\ref{thup}.
\end{proof}

By applying Theorem \ref{dcthm} to the above specialisations, we get the following recursions for the number of nowhere-identity local flows and tensions. This is the main result of this section.

\begin{theorem}\label{flowdc}
    Let $\pG=(\bG,\V,\wV,\F,\wF)$ be a packaged ribbon graph, $e\in E(\bG)$ and $\Gamma$ be a finite group whose irreducible representations have dimensions $n_1,...,n_k$.  
    \begin{enumerate}
        \item For any edge $e$,
        \[q_\Gamma(\pG_\V)=|\Gamma|^{2-\m}q_\Gamma((\pG/e)_\V)-q_\Gamma( (\pG\ba e)_\V).\]
         If $\pG$ has no edges, then 
         \[q_\Gamma(\pG_\V)=\prod_{[v]\in \V} |\Gamma|^{-1}\sum_{k}n_k^{1+|[v]|-\wV([v])}. \]
        \item For any edge $e$,
        \[p_\Gamma(\pG_\F)=|\Gamma|^{2-\n}p_\Gamma((\pG\ba e)_\F)-p_\Gamma((\pG/e)_\F).\]
         If $\pG$ has no edges, then 
         \[p_\Gamma(\pG_\F)=\prod_{[f]\in \F} |\Gamma|^{-1}\sum_{k}n_k^{1+|[f]|-\wF([f])}. \]
    \end{enumerate}    
  \end{theorem}

\begin{proof}
   The first item follows by Theorem \ref{flowspecial} with the specialisations of Theorem \ref{dcthm}, and using the fact that $v(G(\bG;\V))=v(G(\bG\ba e;\V))=v(G(\bG/ e;\V'))$ when $\mu(e)=1$, and $v(G(\bG;\V))=v(G(\bG\ba e;\V))=v(G(\bG/e;\V'))+1$ when $\mu(e)=2$. The case when $\pG$ has no edges follows directly from the definition of $q_\Gamma(\pG_\V)$ as the only edge subset of $E(\bG)$ is the empty set.

    By a similar argument, one can also obtain the second half of the theorem. We omit the details.
\end{proof}

One of the anonymous referees observed that the recurrences for the number of (nowhere-identity) local flows and tensions of a ribbon graph (with no packaging) can partially
be argued directly in a similar way to (nowhere-zero) flows and tensions for graphs. However, this cannot be done for all edge types which is where the lifting to packaged ribbon graphs comes
into its own.

\begin{example}
    We will use Theorem \ref{flowdc} to find $q_\Gamma(\pG_\V)$ for the initial packaged ribbon graph $\pG=(\bG,\V,\wV,\F,\wF)$ depicted in Figure \ref{PolyDC} (i.e. a 2-cycle interlaced with a loop, every partition block has size one, and every block is mapped to zero). Let $\Gamma=D_6$ be the dihedral group of order 6, so $\Gamma$ has two 1-dimensional representations and one 2-dimensional representation. Using Theorem \ref{flowdc} and the deletion-contraction as depicted in Figure~\ref{PolyDC}, 
    \begin{align*}
        q_\Gamma(\pG_\V)&= -6^{-2}(2+2^{2})^2+2\cdot 6^{-1}(2+2^2) -(2+2^2)+ 6^{-1}(2+2^2)^2-2(2+2^2)  +6(2+2^0)\\
        &= -1 + 2 - 6 + 6 - 12 + 18 = 7. 
    \end{align*}
\end{example}

\section*{Acknowledgements}
M.T. would like to thank Guus Regts for helpful discussions. This work will form part of the second author's PhD thesis, which will include additional analysis of the non-orientable case. 

\section*{Open access and research data}
For the purpose of open access, the authors have applied a Creative Commons
Attribution (CC BY) licence to any Author Accepted Manuscript version arising.
No underlying data is associated with this article.
There are no conflicts of interest.

\bibliographystyle{abbrv} 
\bibliography{delcon.bib}
\end{document}